\documentclass[10pt]{article}
\usepackage{mathrsfs}
\usepackage{amsmath,amscd,amsthm,amsfonts}
\oddsidemargin=0.5cm\topmargin=-1cm \numberwithin{equation}{section}

\theoremstyle{plain}
\newtheorem{theorem}{Theorem}[section]
\newtheorem{corollary}[theorem]{Corollary}
\newtheorem{definition}[theorem]{Definition}
\newtheorem{lemma}[theorem]{Lemma}

\newtheorem{proposition}[theorem]{Proposition}
\newtheorem{remark}[theorem]{Remark}

\newtheorem{example}[theorem]{Example}
\newtheorem*{Theorem 1}{Theorem 1}
\newtheorem*{Theorem 2}{Theorem 2}
\newtheorem*{Question 1}{Question 1}
\newtheorem*{Question 2}{Question 2}
\newtheorem*{Prop}{Proposition}

\begin{document}
\vskip 5.1cm
\title{Entropy Formula for Random $\mathbb{Z}^k$-actions\footnotetext{
The author is supported by  NSFC(No:11371120), NSFHB(No:A2014205154), BRHB(No:BR2-219) and GCCHB(No:GCC2014052). }}
 \author {Yujun Zhu\\
\small {College of Mathematics and Information Science, and}\\
   \small {Hebei Key Laboratory of Computational Mathematics and Applications,}\\
\small {Hebei Normal University, Shijiazhuang, Hebei, 050024, P.R.China}}
\date{}
\maketitle
\begin{center}
\begin{minipage}{130mm}
{\bf Abstract}: In this paper, entropies, including measure-theoretic entropy and topological entropy, are considered for  random $\mathbb{Z}^k$-actions which are generated by random compositions of
the generators of $\mathbb{Z}^k$-actions. Applying Pesin's theory for commutative diffeomorphisms we obtain a measure-theoretic entropy formula of $C^{2}$ random $\mathbb{Z}^k$-actions via the Lyapunov spectra of the generators.
Some formulas and bounds of topological entropy for certain random $\mathbb{Z}^k$(or $\mathbb{Z}_+^k$ )-actions generated by more general maps, such as Lipschitz maps, continuous maps on finite graphs and $C^{1}$ expanding maps, are also obtained. Moreover, as an application, we give a formula of Friedland's entropy for certain $C^{2}$ $\mathbb{Z}^k$-actions.
\end{minipage}
\end{center}

\section{Introduction}

It is well known that entropies, including measure-theoretic entropy and topological entropy, are important invariants which play important roles in the study of complexity of the dynamical systems. In deterministic systems, one usually consider the behavior of the iterations of a single transformation. In this paper we are concerned with the entropies of random $\mathbb{Z}^k(k\ge1)$-actions which are generated by random compositions of the generators of $\mathbb{Z}^k$-actions.

Throughout this paper, we assume that $(X, d)$  is a compact metric space and $C^0(X, X)$ is the space of continuous maps on $X$ equipped with the $C^0$-topology. When $X=M$ is a closed (i.e., compact and boundaryless) $C^{\infty}$ Riemannian manifold, we denote $C^r(X, X)$ or $C^r(M, M)(r\geq 1)$ the space of $C^r$ maps equipped with the $C^r$-topology.

Let $\alpha:\mathbb{Z}^k\longrightarrow C^r(X, X)$ ($r\geq0$) be a $C^r$ $\mathbb{Z}^k$-action on $X$. We denote the collection of generators of $\alpha$ by
\begin{equation}\label{generator}
\mathcal{G}_{\alpha}=\{f_i=\alpha(\vec{e}_i):1\leq i\leq k\},
\end{equation}
where $\vec{e}_i=(0,\cdots,\stackrel{(i)}{1},\cdots,0)$ is the standard $i$-th generator of $\mathbb{Z}^k$.  Clearly, all $f_i, 1\le i\le k$, are homeomorphisms when $r=0$, and are $C^r$ diffeomorphisms when $r\geq 1$. Let
\begin{equation}\label{Omega}
\Omega=\mathcal{G}_{\alpha}^{\mathbb{Z}}=\prod_{-\infty}^\infty \mathcal{G}_{\alpha}
\end{equation}
be the infinite product of $\mathcal{G}_{\alpha}$, endowed with the product topology and the product Borel $\sigma$-algebra $\mathcal{A}$,
and let $\sigma$ be the left shift operator on $\Omega$ which is defined by $(\sigma\omega)_n=\omega_{n+1}$ for $\omega=(\omega_n)\in \Omega$. Given $\omega=(\omega_n)\in \Omega$,
we write $f_{\omega}=\omega_0$ and
$$
f_{\omega}^n:=\left\{\begin{array}{ll}
f_{\sigma^{n-1}\omega}\circ\cdots\circ f_{\sigma\omega}\circ f_{\omega} \quad &\;\;n>0\\
id \quad &\;\;n=0\\
f_{\sigma^{n}\omega}^{-1}\circ\cdots\circ f_{\sigma^{-2}\omega}^{-1}\circ f_{\sigma^{-1}\omega}^{-1} \quad &\;\;n<0.
\end{array}\right.
$$
Clearly, each $\omega$ induces a nonautonomous dynamical system generated by the sequence of maps $\{f_{\sigma^i\omega}\}_{i\in \mathbb{Z}}$. There is a natural skew product transformation $\Psi: \Omega\times X\longrightarrow \Omega\times X$ over $(\Omega, \sigma)$ which is defined by
\begin{equation}\label{Psi}
\Psi(\omega, x)=(\sigma\omega, f_{\omega}(x)).
\end{equation}

Let $\alpha:\mathbb{Z}^k\longrightarrow C^r(X, X)$ ($r\geq0$) be a $C^r$ $\mathbb{Z}^k$-action on $X$ and $\nu$ a probability measure on $\mathcal{G}_{\alpha}$. We can define a probability
measure $\mathbf{P}_{\nu}=\nu^{\mathbb{Z}}$ on $\Omega$ which is invariant and ergodic with respect
to $\sigma$. By the induced $C^r(r\geq0)$ \emph{random $\mathbb{Z}^k$-action $f$ over $(\Omega, \mathcal{A}, \mathbf{P}_{\nu}, \sigma)$} we  mean the  system generated by the randomly composed maps $f_{\omega}^n$, $n\in \mathbb{Z}$. It is also called a $C^r $  \emph{i.i.d. (i.e., independent and identically distributed) random $\mathbb{Z}^k$-action}. We are interested in dynamical behaviors of these actions
for $\mathbf{P}_{\nu}$-a.e. $\omega$ or on the average on $\omega$.

When replacing ``$\mathbb{Z}^k$" and ``$n\in \mathbb{Z}$" in the above definition of random $\mathbb{Z}^k$-actions by ``$\mathbb{Z}_+^k$"(here $\mathbb{Z}_+=\{0,1,2,\cdots\}$) and ``$n\in \mathbb{Z}_+$" respectively, we can define random $\mathbb{Z}_+^k$-actions similarly.

The main task of this paper is to investigate the entropies, including measure-theoretic entropy and topological entropy, of random $\mathbb{Z}^k$(or $\mathbb{Z}_+^k$ )-actions. A measure-theoretic entropy formula of $C^{2}$ random $\mathbb{Z}^k$-actions is given in Section 2. Some formulas and bounds of topological entropy for certain random $\mathbb{Z}^k$(or $\mathbb{Z}_+^k$ )-actions are obtained in Section 3. Moreover, as an application, we give a formula of Friedland's entropy for certain $C^{2}$ $\mathbb{Z}^k$-actions in Section 4. In the remaining of this section, we introduce the basic notions and state the main results.

\subsection{Measure-theoretic entropy of random $\mathbb{Z}^k$-actions}

In this subsection and the next subsection, we consider the measure-theoretic entropy and topological entropy of random $\mathbb{Z}^k$-actions respectively.  The basic notions are derived from Kifer \cite{Kifer} and Liu \cite{Liu} in which the ergodic theory of general random dynamical systems are systematically investigated. We can see \cite{Burton} for some progress in the research of entropy for random $\mathbb{Z}^k$-actions.

Let $\alpha:\mathbb{Z}^k\longrightarrow C^r(X, X)$ ($r\geq0$) be a $C^r$ $\mathbb{Z}^k$-action on $X$.
A Borel probability measure $\mu$ on $X$ is called  \emph{$\alpha$-invariant} (resp. \emph{ergodic}) if
 it is invariant (resp. ergodic) with respect to each $f_i,1\le i\le k$.  Let $\nu$ be a probability measure on $\mathcal{G}_{\alpha}$ and $f$ the induced random $\mathbb{Z}^k$-action over $(\Omega, \mathcal{A}, \mathbf{P}_{\nu}, \sigma)$. A Borel probability measure $\mu$ on $X$ is called \emph{$f$-invariant} if
$$
\int_{\Omega} \mu(f_{\omega}^{-1}A)d\mathbf{P}_{\nu}(\omega)=\mu(A)
$$
for all Borel $A\subset X$. Clearly, an $\alpha$-invariant measure must be $f$-invariant.

Let $\mathcal{P}=\{P_k\}$ be a finite or countable partition of $\Omega\times X$ into measurable sets then
$\mathcal{P_{\omega}}=\{P_{k,\omega}\}$, here $P_{k,\omega}=\{x\in X : (\omega, x)\in P_k\}$, is a partition of $X$.

\begin{definition}
Let $\alpha:\mathbb{Z}^k\longrightarrow C^r(X, X)$ ($r\geq0$) be a $C^r$ $\mathbb{Z}^k$-action on $X$ and $\nu$ a probability measure on $\mathcal{G}_{\alpha}$. Let $f$ be the induced random $\mathbb{Z}^k$-action over $(\Omega, \mathcal{A}, \mathbf{P}_{\nu}, \sigma)$ and $\mu$ an $f$-invariant measure. For a finite or countable Borel partition $\mathcal{P}$ of $\Omega\times X$,  the limit
\begin{equation}\label{measentr}
h_{\mu}(f,\mathcal{P}):=\lim_{n\rightarrow\infty}\frac{1}{n}\int_{\Omega}
H_{\mu}
\big(\bigvee_{i=0}^{n-1}(f_{\omega}^i)^{-1}\mathcal{P_{\omega}}
\big)\;d\mathbf{P}_{\nu}(\omega),
\end{equation}
where $H_{\mu}(\mathcal{Q}):=-\sum\limits_{A\in\mathcal{Q}}\mu(A)\log\mu(A)$ for a finite or countable
partition $\mathcal{Q}$ of $X$, exists. The number
$$
h_{\mu}(f):=\sup_{\mathcal{P}}h_{\mu}(f,\mathcal{P}),
$$
where $\mathcal{P}$ ranges over all finite  or countable partitions of $\Omega\times X$, is called the
\emph{measure-theoretic entropy} of $f$.
\end{definition}
In fact, since $(\Omega, \mathcal{A}, \mathbf{P}_{\nu}, \sigma)$ is ergodic and measurably invertible, taking integration in
(\ref{measentr}) is extraneous and the limit exists and is constant for $\mathbf{P}_{\nu}$-a.e. $\omega$. However, when considering a random $\mathbb{Z}_+^k$-action (here $\mathbb{Z}_+=\{0,1,2,\cdots\}$) over $(\Omega, \mathcal{A}, \mathbf{P}_{\nu}, \sigma)$, where $\Omega=\mathcal{G}_{\alpha}^{\mathbb{Z}_+}=\prod_0^\infty \mathcal{G}$ for the set $\mathcal{G}$ of generators with respect to a $\mathbb{Z}_+^k$-action $\alpha:\mathbb{Z}_+^k\longrightarrow C^r(X, X)$, we must keep the integration in the corresponding definition.

\begin{remark}\label{finitepartition}
Usually, we often use the following equivalent definition of $h_{\mu}(f)$,
$$
h_{\mu}(f)=\sup_{\mathcal{P}}h_{\mu}(f,\mathcal{P}),
$$
where $\mathcal{P}$ ranges over all finite  partitions of $\Omega\times X$ in form of $\mathcal{Q}\times \Omega$ in which $\mathcal{Q}$ is a finite partition of $X$.
\end{remark}

It is well known that Pesin's entropy formula plays an important role in smooth ergodic theory, including the deterministic case and the random case. For the deterministic case we refer to \cite{Ruelle}, \cite{Pesin}, \cite{Mane} \cite{Ledrappier-Strelcyn} and \cite{Ledrappier-Young}, and for the random case, we refer to \cite{Bahnmuller0}, \cite{Bahnmuller}, \cite{Liu}  and \cite{Liu01}. It gives an explicit formula relating the measure-theoretic entropy and Lyapunov exponents in corresponding settings. In particular, for a general $C^2$ i.i.d. random dynamical system $f$ which is defined by replacing ``$\mathcal{G}_{\alpha}$" by ``Diff$^2(X)$" (i.e., the set of $C^2$ diffeomorphisms)  in the above notations for random  $\mathbb{Z}^k$-action, if
$$
\log^+\|f_{\omega}\|_{C^2}, \log\inf_{x\in M}|\det D_xf_{\omega}| \in L^1(\Omega, \mathbf{P}_{\nu}),
$$
then for any $f$-invariant measure $\mu$ which is absolutely continuous with respect to the Lebesgue measure of $X$, the following entropy formula
\begin{equation}\label{Liufomula}
h_{\mu}(f)=\int_X \sum_{\lambda_i(x)>0}\lambda_i(x)d_i(x)d\mu(x)
\end{equation}
holds, where
$\{(\lambda_i(x),d_i(x))\}$ is the essentially non-random Lyapunov spectrum of $f$.   Though the Lyapunov spectrum of $f$ is essentially non-random, we can not expect to give some explicit relation between it and the spectra of the elements in Diff$^2(X)$.
Now, for  a $C^2$  random $\mathbb{Z}^k$-action $f$,  $\mathcal{G}_{\alpha}$ consists of finite elements which are pairwise commutative. A natural question is
 \begin{Question 1}
Is there an entropy formula via the Lyapunov spectra of the generators ?
\end{Question 1}

The main aim of this paper is to give a positive answer of the above question under the assumption that the measure $\mu$ is $\alpha$-invariant and is absolutely continuous with respect to the Lebesgue measure of $X$. Hence the measure-theoretic entropy of such $C^2$  random $\mathbb{Z}^k$-actions is, in a sense,  easer to be calculated.

Section 2 is devoted to the proof of a measure-theoretic entropy formula, i.e., the main result of this paper.

\begin{Theorem 1}\label{Thm 1}
Let $\alpha:\mathbb{Z}^k\longrightarrow C^2(M, M)$ be a $C^2$ $\mathbb{Z}^k$-action on a $d$-dimensional closed Riemannian manifold $M$ and $\nu$ a probability measure on $\mathcal{G}_{\alpha}$. Then for the induced random $\mathbb{Z}^k$-action $f$ over $(\Omega, \mathcal{A}, \mathbf{P}_{\nu}, \sigma)$ and any $\alpha$-invariant measure $\mu$ which is absolutely continuous with respect to the Lebesgue measure $m$, we have the following entropy formula
\begin{equation}\label{entropyformula1}
h_{\mu}(f)=\int_M\max_{J\subset \{1,\cdots,s(x)\}}\sum_{j\in J}\sum_{i=1}^k \nu_i d_j(x)\lambda_{i,j}(x)d\mu(x),
\end{equation}
where $\nu_i=\nu(f_i)$ and $\{(\lambda_{i,j}(x),d_j(x)): 1\le i\le k, 1\le j\le s(x)\}$ is the spectrum of $\alpha$  (see Definition \ref{spectrum} for the precise definition).

In particular, if $\mu$ is $\alpha$-ergodic, then we have
\begin{equation}\label{entropyformula2}
h_{\mu}(f)=\max_{J\subset \{1,\cdots,s\}}\sum_{j\in J}\sum_{i=1}^k \nu_i d_j\lambda_{i,j}.
\end{equation}
(Note that when $\mu$ is $\alpha$-ergodic, then $\lambda_{i,j}(x),d_j(x)$ and $s(x)$ are all constant and we denote them by $\lambda_{i,j},d_j$ and $s$, respectively.)
\end{Theorem 1}

The strategy to prove Theorem 1 is to adapt the proof of Ruelle's inequality in Theorem S2.13 of \cite{Katok} and the proof of the reverse inequality in section 13 of \cite{Mane1} for any $C^2$ diffeomorphisms to oue case.  In the proof, Pesin's theory for commutative $C^2$ diffeomorphisms developed by Hu (\cite{Hu}) and Birkhoff's Ergodic Theorem play important roles.

\subsection{Topological entropy of random $\mathbb{Z}^k$ and $\mathbb{Z}_+^k$-actions}
Let $\alpha:\mathbb{Z}^k$ (resp. $\mathbb{Z}_+^k)\longrightarrow C^r(X, X)$ ($r\geq0$) be a $C^r$ $\mathbb{Z}^k$ (resp. $\mathbb{Z}_+^k)$-action on $X$, $\nu$ a probability measure on $\mathcal{G}_{\alpha}$ and $f$  the induced random $\mathbb{Z}^k$ (resp. $\mathbb{Z}_+^k)$-action over $(\Omega, \mathcal{A}, \mathbf{P}_{\nu}, \sigma)$.  Define a family of metrics $\{d^n_\omega: n\in\mathbb{Z_+},\omega\in
\Omega\}$ on $X$ by
$$
d^n_\omega(x, y)=\max_{0\leq i\leq n-1}d(f^i_\omega(x), f^i_\omega(y))
$$
for $x, y \in X$.
Let $\varepsilon>0, \omega\in \Omega$ and $K$ be a subset of $X$. A set $F\subset X$ is said
to be an $(\omega,n,\varepsilon)$-\emph{spanning set} of $K$ if
for any $x\in K$ there exists $y\in F$ such that
$d_\omega^n(x,y)\leq \varepsilon.$ Let $r(\omega,n,\varepsilon,K)$
denote the smallest cardinality of any
$(\omega,n,\varepsilon)$-spanning set of $K$. A subset $E\subset
K$ is said to be an $(\omega,n,\varepsilon)$-\emph{separated set}
of $K$ if $x,y\in E,x\neq y$ implies $d_\omega^n(x,y)>
\varepsilon$.
 Let $s(\omega,n,\varepsilon,K)$
denote the largest cardinality of any
$(\omega,n,\varepsilon)$-separated set of $K$.
\begin{definition}
Let $f$ be a random $\mathbb{Z}^k$ (or $\mathbb{Z}_+^k)$-action over $(\Omega, \mathcal{A}, \mathbf{P}_{\nu}, \sigma)$ as above. Then the limit
\begin{equation}\label{hphi}
h(f):=\lim_{\varepsilon \rightarrow
0}\limsup_{n\rightarrow\infty} \frac{1}{n}\int_{\Omega} \log
 r(\omega,n,\varepsilon,X)\;d\mathbf{P}_{\nu}(\omega)
\end{equation}
exists. The number $h(f)$ is called the
\emph{topological entropy} of $f$.
\end{definition}

From \cite{Kifer1} we have
\begin{equation}\label{hphi1}
h(f)=\int_{\Omega}\lim_{\varepsilon \rightarrow
0}\limsup_{n\rightarrow\infty} \frac{1}{n} \log
 r(\omega,n,\varepsilon,X)\;d\mathbf{P}_{\nu}(\omega).
\end{equation}
Moreover, we can get the same value of $h(f)$ when we replace ``$\limsup$" by ``$\liminf$", or replace   ``$r(\omega,n,\varepsilon,X)$" by ``$s(\omega,n,\varepsilon,X)$" in (\ref{hphi}) and (\ref{hphi1}). If we denote the topological entropy of the nonautonomous dynamical system $\{f_{\sigma^i\omega}\}_{i\in \mathbb{Z}_+}$  by $h(f, \omega)$, then by
\cite{Kolyada},
$$
h(f, \omega)=\lim_{\varepsilon \rightarrow
0}\limsup_{n\rightarrow\infty} \frac{1}{n} \log
 r(\omega,n,\varepsilon,X),
$$
and hence
\begin{equation}\label{topint}
h(f)=\int_{\Omega} h(f, \omega)\;d\mathbf{P}_{\nu}(\omega).
\end{equation}

Generally, we have the following inequality relating topological entropy and measure-theoretic entropy,
\begin{equation}\label{VP}
h(f)\geq \sup_{\mu}h_{\mu}(f),
\end{equation}
where $\mu$ ranges over all $f$-invariant measures. (\ref{VP}) is called a \emph{variational principle}. For more general random dynamical systems, (\ref{VP})  may hold as an equality, see Theorem 3.1 in \cite{Liu01} for example. In general, to give an explicit formula of topological  entropy, even for a deterministic system, is not an easy work. One often give the lower or upper bounds of topological entropy via various topological quantities, such as degrees and volume growth rates, see \cite{Katok} and \cite{Franks} for example.

In section 3, we give some formulas and bounds of topological entropy for certain random $\mathbb{Z}^k$(or $\mathbb{Z}_+^k$)-actions generated by more general maps, such as Lipschitz maps (Proposition \ref{Lips}), continuous maps on finite graphs and $C^{1}$ expanding maps (Proposition \ref{twohi}). By (\ref{topint}),  topological entropy $h(f)$ is the integral of that of nonautonomous dynamical systems $h(f,\omega), \omega\in \Omega$. By some known results about  topological entropy of nonautonomous dynamical systems and Birkhoff's Ergodic Theorem, we get the following formulas and bounds.

\begin{Prop}
 Let $\alpha:\mathbb{Z}_+^k\longrightarrow C^r(X, X)$ ($r\geq0$) be a $C^r$ $\mathbb{Z}_+^k$-action on $X$, $\nu$ a probability measure on $\mathcal{G}_{\alpha}$ with $\nu_i=\nu(f_i)$, and $f$  the induced random $\mathbb{Z}_+^k$-action over $(\Omega, \mathcal{A}, \mathbf{P}_{\nu}, \sigma)$.

 (1) If the generators $f_i, 1\leq i\leq k$, are all Lipschitz maps with the Lipschitz constants $L(f_i), 1\leq i\leq k$, respectively, then we have
$$
 h(f)\leq D(X)\sum_{i=1}^k \nu_i\log L^+(f_i),
$$
 where $L^+(f_i)=\max\{1, L(f_i)\}$ and $D(X)$ is the ball dimension of $X$.

(2) If $X$ is a finite graph and  the generators $f_i, 1\leq i\leq k$, are all homeomorphisms, then $h(f)=0$.

(3) If $X=M$ is a closed oriented Riemannian manifold and  the generators $f_i, 1\leq i\leq k$, are all $C^1$ expanding maps, then we have
$$
 h(f)=\sum_{i=1}^k \nu_i\log |\deg(f_i)|,
$$
 where $\deg(f_i)$ is the degree of $f_i$.
\end{Prop}

\subsection{Friedland's entropy of  $\mathbb{Z}^k$-actions}
Friedland's entropy of  $\mathbb{Z}^k$-actions was introduced by Friedland \cite{Friedland} via the topological entropy of the shift map on the induced orbit space. More precisely,
let  $\alpha:\mathbb{Z}^k\longrightarrow C^r(X, X) (r\ge 0)$ be a $\mathbb{Z}^k$-action on $X$ with the generators $\{f_i\}_{i=1}^k$. Define the \emph{orbit space} of $\alpha$ by
$$
X_{\alpha}=\bigl\{\bar{x}=\{x_n\}_{n\in
{\mathbb{Z}}}\in\prod_{n\in{\mathbb{Z}}}X : \text{ for any }  n\in{\mathbb{Z}}, f_{i_n}(x_n)=x_{n+1}
\text{ for some } f_{i_n}\in\{f_i\}_{i=1}^k \bigr\}.
$$
This is a closed subset of the compact space
$\prod\limits_{n\in{\mathbb{Z}}}X$ and so is again compact. A
natural metric $\bar{d}$ on $X_{\alpha}$ is defined by
\begin{equation}\label{metric}
\bar{d}\bigl(\bar{x},\bar{y}\bigr)=\sum_{n=-\infty}^\infty\frac{d(x_n,y_n)}{2^{|n|}}
\end{equation}
for $\bar{x}=\{x_n\}_{n\in
{\mathbb{Z}}}, \bar{y}=\{y_n\}_{n\in
{\mathbb{Z}}}\in X_{\alpha}$.
We can define a shift map
$$
\sigma_{\alpha}: X_{\alpha}\rightarrow X_{\alpha},\;\sigma_{\alpha}(\{x_n\}_{n\in {\mathbb{Z}}})=\{x_{n+1}\}_{n\in
{\mathbb{Z}}}.
$$
 Thus we have associated a
${\mathbb{Z}}$-action with the ${\mathbb{Z}}^k$-action $\alpha$.

\begin{definition}
\emph{Friedland' s entropy} of a
${\mathbb{Z}}^k$-action $\alpha$ is defined by the topological entropy of the shift
map $\sigma_{\alpha}: X_{\alpha}\rightarrow X_{\alpha}$, i.e.,
\begin{equation}\label{Friedland}
h(\sigma_{\alpha})=\lim_{\varepsilon\rightarrow
0}\limsup_{n\rightarrow\infty}\frac{1}{n}\log
\;s_{\bar{d}}(\sigma_{\alpha},n,\varepsilon, X_{\alpha}),
\end{equation}
where $s_{\bar{d}}(\sigma_{\alpha},n,\varepsilon, X_{\alpha})$ is the largest cardinality of any $(\sigma_{\alpha},n, \varepsilon)$-separated sets of $X_{\alpha}$.
\end{definition}
 Replacing ``$\mathbb{Z}$" by ``$\mathbb{Z}_+$" in the above statement, we can get the definition of Friedland's entropy for
${\mathbb{Z}}_+^k$-actions. Unlike the classical entropy  for $\mathbb{Z}^k$ (or ${\mathbb{Z}}_+^k$)-actions, Friedland's entropy is positive when the generators
 have finite entropy as single transformations. From the known results about Friedland's entropy, we can see that it is not an easy task to compute it, even for some ``simple" examples. In \cite{Friedland}, Friedland conjectured that for a
$\mathbb{Z}_+^2$-action on the circle
$\mathbb{S}^1$ whose generators are
$T_1=px(\mbox{mod } 1)$ and $T_2=qx(\mbox{mod } 1)$, where $p$ and
$q$ are two co-prime integers, its entropy
$$
h(\sigma_{\alpha})=\log(p+q).
$$
Soon afterwards Geller and Pollicott \cite{Geller} answered this
conjecture affirmatively under a weaker condition ``$p, q$ are different
integers greater than 1".
In \cite{Einsiedler} and \cite{Zhu14}, the formulas of Friedland's entropy for the ``linear" $\mathbb{Z}^k$ and ${\mathbb{Z}}_+^k$-actions on the torus via the eigenvalues of the generators are given respectively. We can see that all the known formulas of Friedland's entropy as we mentioned above were obtained for special (``linear") actions on simple manifolds (tori). A natural question is
 \begin{Question 2}
Is there a formula of Friedland's entropy for general smooth $\mathbb{Z}^k$ and ${\mathbb{Z}}_+^k$-actions via the Lyapunov spectra of the generators ?
\end{Question 2}

 In section 4, applying the entropy formula (\ref{entropyformula2}) for random $\mathbb{Z}^k$-actions and use the techniques, especially the variational principle for pressures and Birkhoff's Ergodic Theorem, in \cite{Einsiedler} and \cite{Zhu14}, we give some formulas and bounds of Friedland's entropy for certain $\mathbb{Z}^k$-actions. Hence, the Friedland's entropy formulas in \cite{Geller}, \cite{Einsiedler} and \cite{Zhu14} are all special cases of the following  result.

 \begin{Theorem 2}\label{Thm2}
Let $\alpha:\mathbb{Z}^k\longrightarrow C^2(M, M)$ be a $C^2$ $\mathbb{Z}^k$-action on a  $d$-dimensional closed Riemannian manifold $M$. Let $\mathcal{G}_{\alpha}$ and $\Omega$ be as in (\ref{generator}) and (\ref{Omega}) respectively, and $\Psi$ the skew product transformation as in  (\ref{Psi}). If there is a measure with maximal entropy of $\Psi$ in the form of $\mathbf{P}_{\nu}\times\mu$, where $\mathbf{P}_{\nu}=\nu^{\mathbb{Z}}$ is the product measure of some Borel probability measure $\nu$  on $\mathcal{G}_{\alpha}$ with $\nu_i=\nu(f_i)$ and
$\mu$ is an $\alpha$-invariant measure on $M$ which is absolutely continuous with respect to the Lebesgue measure $m$. Then
\begin{equation}\label{Fried1}
h(\sigma_{\alpha})\le -\sum_{i=1}^k\nu_i\log\nu_i +\int_M\max_{J\subset \{1,\cdots,s(x)\}}\sum_{j\in J}\sum_{i=1}^k \nu_i d_j(x)\lambda_{i,j}(x)d\mu(x).
\end{equation}

Furthermore, if  $\mu$ is $\alpha$-ergodic and for  each pair of generators $f_i$ and $f_j$, $1\le i\neq j\le k$, the set of their  coincidence points, i.e. \emph{Coinc}$(f_i, f_j)$ which is defined by $\{x\in M:f_i(x)=f_j(x)\}$, is of $\mu$ measure zero, then we get the following formula of Friedland's entropy
\begin{equation}\label{Fried3}
h(\sigma_{\alpha})= \displaystyle\max_{J\subset \{1,\cdots,s\}}\log\Big(\sum_{i=1}^k\exp\big(\sum_{j\in J}d_j\lambda_{i,j}\big)\Big).
\end{equation}
\end{Theorem 2}

\section{A measure-theoretic entropy formula of $C^2$ random $\mathbb{Z}^k$-actions}

We first recall some fundamental properties of  $C^2$ $\mathbb{Z}^k$-actions. Let $M$ be a $d$-dimensional closed $C^{\infty}$ Riemannian manifold. We denote by $\langle\!\langle\cdot,\cdot\rangle\!\rangle$,  $\|\cdot\|$ and $d(\cdot,\cdot)$ the inner product, the norm on the tangent spaces and the metric on $M$ induced by the Riemannian metric, respectively.

Let $\alpha:\mathbb{Z}^k\longrightarrow C^2(M, M)$ be a $C^2$ $\mathbb{Z}^k$-action on $M$  with the generators $f_i,1\le i\le k$, as in (\ref{generator}) and $\mu$ be an $\alpha$-invariant measure. By the Multiplicative Ergodic Theorem (\cite{Oseledec}), for each $f_i$ there exists a measurable set $\Gamma_i$ with $f_i(\Gamma_i)=\Gamma_i$ and $\mu(\Gamma_i)=1$, such that
for any $x\in \Gamma_i$, there exist a decomposition
$$
T_xM=\bigoplus_{j=1}^{r(x,f_i)}E_j(x)
$$
into subspaces $E_j(x)$ of dimension $d_j(x,f_i)$ (where
$\sum\limits_{j=1}^{r(x,f_i)}d_j(x,f_i)=d$), and numbers
$\lambda_1(x,f_i)<\cdots<\lambda_{r(x,f_i)}(x,f_i)$ which satisfy the following properties:

(1) for $1\le j\le r(x,f_i), v\in E_j(x)\setminus \{0\}$,
$$
\lim_{n\longrightarrow \pm \infty}\frac{1}{n}\log\|Df^n_i(x)v\|=\lambda_j(x,f_i);
$$

(2) $E_j(x), d_j(x,f_i$) and $\lambda_j(x,f_i)$ all measurably depend on $x$ and the following invariance properties
$$
Df_i(x)E_j(x)=E_j(f_i(x))\;\;\text{and}\;\; \lambda_j(f_i(x),f_i)=\lambda_j(x,f_i),
$$
hold for each $1\le i\le k, 1\le j\le r(x,f_i)$.

The above numbers $\lambda_1(x,f_i),\cdots, \lambda_{r(x,f_i)}(x,f_i)$ are called the \emph{Lyapunov exponents} of $f_i$ at $x$, and the collection $\{(\lambda_j(x,f_i),d_j(x,f_i)): 1\le j\le r(x,f_i), x\in \Gamma_i\}$ is called the \emph{spectrum} of $f_i$. For each $1\le i\le k$ and each $x\in \Gamma_i$, denote
$$
E^s(x, f_i)=\bigoplus_{\lambda_j(x,f_i)<0}E_j(x)\;\text{ and }\;E^u(x,f_i)=\bigoplus_{\lambda_j(x,f_i)>0}E_j(x).
$$

In \cite{Hu}, Hu discussed some ergodic properties of $C^2$ $\mathbb{Z}^2$-actions concerning Lyapunov exponents and entropies. He gave a version of Pesin's theory for this case. More precisely, by the commutativity of the generators, he gave a family of refined decompositions of the above $\{E_j(x)\}$ into subspaces related to the Lyapunov exponents of both of the generators $f_1$ and $f_2$,  then constructed a family of Lyapunov charts and applied them to obtain the subadditivity of the entropies. We first introduce some fundamental properties for $C^2$ $\mathbb{Z}^k$-actions,  they all derive from \cite{Hu}.

\begin{proposition}\label{Lyapunovexponents}
Let $\alpha:\mathbb{Z}^k\longrightarrow C^2(M, M)$ be a $\mathbb{Z}^k$-action on $M$  with the generators $f_i,1\le i\le k$ and $\mu$ be an $\alpha$-invariant measure. Then there exists a measurable set $\Gamma\subset \Gamma_i$ with $f_i(\Gamma)=\Gamma$ for each $i$ (in this case we call $\Gamma$ is $\alpha$-\emph{invariant}) and $\mu(\Gamma)=1$, such that for any $x\in \Gamma$, there exists a decomposition of tangent space into subspaces
$$
T_xM=\bigoplus_{j_1=1}^{r(x,f_1)}\cdots\bigoplus_{j_k=1}^{r(x,f_k)}E_{j_1,\cdots,j_k}(x)
$$
satisfying the following properties:

(1) if $E_{j_1,\cdots,j_k}(x)\neq\{0\}$, then for $0\neq v\in E_{j_1,\cdots,j_k}(x)$ and $1\le i\le k$,
$$
\lim_{n\longrightarrow \pm \infty}\frac{1}{n}\log\|Df_i^n(x)v\|=\lambda_{j_i}(x,f_i);
$$

(2) for each $E_{j_1,\cdots,j_k}(x)$, we have the following invariance properties
$$
Df_i(x)E_{j_1,\cdots,j_k}(x)=E_{j_1,\cdots,j_k}(f_i(x))\;\;\text{and}\;\; \lambda_{j_i}(f_{i'}(x),f_i)=\lambda_{j_i}(x,f_i),
$$
where $1\le i,i'\le k$.
\end{proposition}

Notice that the subspaces $E_{j_1,\cdots,j_k}(x)$, $1\le j_i\le r(x,f_i), 1\le i\le k$, of $T_xM$ may not be pairwise different, and the Lyapunov exponents of $f_i$ with respect to different subspaces may be coincide. For simplicity of the notations, we relabel these subspaces by $F_j(x), 1\le j\le s(x)$, such that
$$
T_xM=\bigoplus_{j=1}^{s(x)} F_j(x)
$$
in which $F_j(x)\neq\{0\}$ for each $1\le j\le s(x)$, and rename the corresponding Lyapunov exponents of $f_i$ with respect to $F_j(x)$  by $\lambda_{i,j}(x)$. If denote $d_j(x)=\dim F_j(x)$, then clearly $\sum_{j=1}^{s(x)}d_j(x)=d$. Therefore, the items (1) and (2) in Proposition \ref{Lyapunovexponents} become

(1$'$) for $0\neq v\in F_j(x), 1\le j\le s(x)$,
$$
\lim_{n\longrightarrow \pm \infty}\frac{1}{n}\log\|Df_i^n(x)v\|=\lambda_{i,j}(x);
$$
and

(2$'$)  for each $F_j(x)$, we have the following invariance properties
$$
Df_i(x)F_j(x)=F_j(f_i(x))\;\;\text{and}\;\; \lambda_{i,j}(f_{i'}(x))=\lambda_{i,j}(x),
$$
where $1\le i,i'\le k$.

\begin{definition}\label{spectrum}
We call the collection
$$
\{(\lambda_{i,j}(x),d_j(x)): 1\le i\le k, 1\le j\le s(x), x\in \Gamma\}
$$
 the \emph{spectrum} of $\alpha$.
 \end{definition}
 When the measure $\mu$ in Proposition \ref{Lyapunovexponents} is ergodic with respect to $\alpha$, then the above $\lambda_{i,j}(x),d_j(x)$ and $s(x)$ are all constant a.s., which are then denoted by $\lambda_{i,j},d_j$ and $s$ respectively.

Before starting to prove Theorem 1, we first give some remarks and examples about it.

\begin{remark}\label{Pesin}
Let $\alpha:\mathbb{Z}^k\longrightarrow C^2(M, M)$ be a $C^2$ $\mathbb{Z}^k$-action on $M$ and $\nu$ a probability measure on $\mathcal{G}_{\alpha}$. Let $f$ be the induced random $\mathbb{Z}^k$-action over $(\Omega, \mathcal{A}, \mathbf{P}_{\nu}, \sigma)$ and $\mu$ a $f$-invariant measure which is absolutely continuous with respect to the Lebesgue measure $m$.

(1)  When $k=1$,  $f$ is no longer a random system but a deterministic system generated by a single diffeomorphism, and  (\ref{entropyformula1}) is exactly the well-known Pesin's entropy formula as in \cite{Pesin},
$$
h_{\mu}(f)=\int_M \sum_{\{\lambda_j(x):\lambda_j(x)>0\}}\lambda_j(x)d_j(x)d\mu(x),
$$
where
$\{(\lambda_j(x),d_j(x))\}$ is the Lyapunov spectrum  of $f$.

(2) From the formula (\ref{entropyformula1}), we can give the relation between the entropy of $f$ and those of the generators as follows
\begin{equation}\label{entropyrelation}
h_{\mu}(f)\le\sum_{i=1}^k\nu_ih_{\mu}(f_i).
\end{equation}
Particularly, if
$$
E^s(x, f_i)\cap E^u(x, f_j)=\{0\}, \mu-a.e. x \;\text{ for all } 1\le i, j\le k,
$$
then the equality in (\ref{entropyrelation}) holds.  Moreover, the equality in (\ref{entropyrelation}) also holds in this particular case  whenever $\mu$ is an ``SRB" measure, i.e., it has smooth conditional
measures on the unstable manifolds of $f$.

(3) When $\mu$ is $\alpha$-ergodic, (\ref{entropyformula2}) can be regarded as a function of the distribution $\nu$ on the set of generators of $\alpha$. Precisely, let $\mathbb{S}^{k-1}=\{(\nu_1,\cdots,\nu_k):\sum_{t=1}^k\nu_1^2=1\}$ be the $k-1$-dimensional unit sphere and denote
$\mathbb{S}^{k-1}_*=\{(\nu_1,\cdots,\nu_k)\in\mathbb{S}^{k-1}:\nu_i> 0, 1\le i\le k\}$. we can define a function $h_{\mu}(f)$ on $\mathbb{S}^{k-1}_*$ by
$$
h_{\mu}(f)((\nu_1,\cdots,\nu_k))=\max_{J\subset \{1,\cdots,s\}}\sum_{j\in J}\sum_{i=1}^k \nu_i d_j\lambda_{i,j}.
$$
Hence applying the theory of conditional extremum, we can get the extremum of the entropy of $f$, and from the corresponding distributions we can identify the contributions of different generators to the complexity of the action $f$.
\end{remark}

\begin{example}\label{example}
Let $\alpha$ be a ${\mathbb{Z}}^k$-action on the torus
$X=\mathbb{T}^d$ with the generators $\{f_i\}_{i=1}^k$ which are induced by non-singular integer  matrices $\{A_i\}_{i=1}^k$. Since $A_iA_j=A_jA_i$ for any $1\leq i, j\leq k$, then by Theorem A of \cite{Hu}, we can write $\mathbb{R}^d$ as a direct sum of subspaces $V_1\bigoplus \cdots\bigoplus V_s$ which are all invariant under each $A_i, 1\leq i\leq k$, and moreover, for each $1\leq i\leq k$ and $1\leq j\leq s$, the eigenvalues of $A_{i,j}=A_i|_{V_j}$ all have the same norm. For each $1\leq j\leq s$ denote $d_j=\dim V_j$ and $\lambda_{i,j}$ the common norm of the eigenvalues of $A_{i,j}$. In fact, the collection $\{(\log\lambda_{i,j},d_j): 1\le i\le k, 1\le j\le s\}$ is exactly the spectrum of $\alpha$. Therefore, by Theorem 1, for any probability measure $\nu$ on $\mathcal{G}_{\alpha}$ and the corresponding random ${\mathbb{Z}}^k$-action $f$ over $(\Omega, \mathcal{A}, \mathbf{P}_{\nu}, \sigma)$, we have
\begin{equation}\label{entropytorus}
 h(f)=h_{m}(f)=\max_{J\subset \{1,\cdots,s\}}\sum_{j\in J} \sum_{i=1}^k\nu_i d_j \log\lambda_{i,j}
\end{equation}
(recall that $m$ is the Lebesgue measure).

In a particular case, for the ${\mathbb{Z}}^2$-action $\alpha$ on the torus
$\mathbb{T}^2$ with the generators $\{f_1,f_2\}$ which are induced by the hyperbolic automorphisms
$$
A_1=\left(
         \begin{array}{rr}
              2 &1\\
              1 &1
          \end{array} \right)\;\; \mbox{and}\;\;
          A_2=A_1^{-1}=\left(
         \begin{array}{rr}
              1 &-1\\
              -1 &2
          \end{array} \right),
$$
 respectively, and the induced random ${\mathbb{Z}}^2$-action $f$ over $(\Omega, \mathcal{A}, \mathbf{P}_{\nu}, \sigma)$, we have
$$
h(f)=|\nu_1-\nu_2|\log\frac{3+\sqrt{5}}{2},
$$
where $\nu_i=\nu(f_i)$.
\end{example}

In the remaining of this section we will prove the main theorem, i.e., Theorem 1, of this paper. We always assume that $\alpha:\mathbb{Z}^k\longrightarrow C^2(M, M)$ is a $C^2$ $\mathbb{Z}^k$-action on $M$, $\nu$ is a probability measure on $\mathcal{G}_{\alpha}$ with $\nu_i=\nu(f_i)$, $f$ is the induced random $\mathbb{Z}^k$-action over $(\Omega, \mathcal{A}, \mathbf{P}_{\nu}, \sigma)$ and $\mu$ is an $\alpha$-invariant measure which is absolutely continuous with respect to the Lebesgue measure $m$.

It is well known that for any $C^2$ diffeomorphism $g$ on $M$ and any invariant measure $\mu$ which is absolutely continuous with respect to $m$, we have the so-called Pesin's entropy formula
 $$
 h_{\mu}(g)=\int_M\sum_{\lambda_j(x)>0}d_j(x)\lambda_j(x)d\mu(x),
 $$
 where $\{(\lambda_j(x),d_j(x))\}$ is the spectrum of $g$. For the proof of the inequality
 $$
  h_{\mu}(g)\leq\int_M\sum_{\lambda_j(x)>0}d_j(x)\lambda_j(x)d\mu(x),
 $$
 which is called Ruelle's inequality, we refer to Theorem S2.13 of \cite{Katok}, and for the proof of the reverse inequality we refer to section 13 of \cite{Mane1}. We will adapt the methods in \cite{Katok} and \cite{Mane1} to get the desired entropy formula (\ref{entropyformula1}) for random $\mathbb{Z}^k$-actions.

\begin{lemma}[Lemma 1 of \cite{Bahnmuller0}]\label{powerrule}
 (1) Define $f^n$ by $(f^n)_{\omega}^i=f^{ni}_{\omega}, i\in \mathbb{Z}, \omega\in \Omega$, for $n\in \mathbb{N}$ (here $\mathbb{N}=\{1,2,\cdots\}$). Then for all
$n\in \mathbb{N}$ one has
\begin{equation}\label{powerrule1}
 h_{\mu}(f^n)=nh_{\mu}(f).
\end{equation}

(2)  If $\{\mathcal{Q}_p\}$ is a sequence of finite partitions of $M$ with $\lim_{p\rightarrow +\infty} diam \mathcal{Q}_p=0$, then
\begin{equation}\label{limit}
h_{\mu}(f)=\lim_{p\rightarrow +\infty}h_{\mu}(f, \mathcal{Q}_p).
\end{equation}
\end{lemma}

Suppose $\varepsilon\in (0,1)$ and $\rho_{\varepsilon}:\Omega \times M\longrightarrow (0, \epsilon)$ is a measurable function. Given $\omega\in \Omega, x\in M$ and $n>0$, define two subsets by
 \begin{equation}\label{Bowenepsball}
B_n(\omega, \varepsilon, x)=\{y\in M:d_{\omega}^n(x,y)\leq \varepsilon\}
\end{equation}
and
 \begin{equation}\label{Bowenrhoball}
B_n(\omega, \rho_{\varepsilon}, x)=\{y\in M:d(f_{\omega}^i(x),f_{\omega}^i(y))\leq \rho_{\varepsilon}(\sigma^i\omega , f_{\omega}^i(x)), 0\le i\le n-1\}.
\end{equation}

\begin{lemma}\label{localentropy}
 Let $\mathcal{H}=\{\rho_{\varepsilon}:\varepsilon\in (0,1)\}$ be a family of measurable functions on $\Omega \times M$ such that for any $\omega\in\Omega$, $\rho_{\varepsilon}(\omega,\cdot)$ monotonically decreases as $\varepsilon\longrightarrow 0$. Then
\begin{equation}\label{hmu}
h_{\mu}(f)\geq\int_M \lim_{\varepsilon\longrightarrow 0}\limsup_{n\rightarrow\infty} -\frac{1}{n}\log m(B_n(\omega, \rho_{\varepsilon},x))d \mu(x), \;\;\mathbf{P}_{\nu}-a.e. \; \omega\in \Omega.
\end{equation}
\end{lemma}
\begin{proof}
Let $\Psi:\Omega\times M\longrightarrow \Omega\times M$ be the induced skew product transformation which is defined by $\Psi(\omega,x)=(\sigma\omega,f_{\omega}(x))$ as in (\ref{Psi}). Clearly, $\mathbf{P}_{\nu}\times\mu$ is an invariant measure of $\Psi$.
Given a finite  or countable partition $\mathcal{P}$ of $\Omega\times M$, for $(\omega, x)\in \Omega\times M$ and $n\ge 0$ denote $\mathcal{P}_{\omega,n}(x)$ the element of $\mathcal{P}_{\omega,n}=\bigvee_{i=0}^{n-1}(f_{\omega}^i)^{-1}\mathcal{P_{\omega}}
$ that contains $x$, specially, denote $\mathcal{P}_{\omega}(x)=\mathcal{P}_{\omega,0}(x)$.
By the Shannon-McMillan-Breiman Theorem for random dynamical systems (see Proposition 2.1 of \cite{Zhu08} and Theorem 1.14 of \cite{Kifer2}, for example), we have
\begin{equation}\label{SMB}
h_{\mu}(f,\mathcal{P})=\int_M \limsup_{n\rightarrow\infty} -\frac{1}{n}\log \mu\big(\mathcal{P}_{\omega,n}(x)\big)d\mu(x),\;\;\mathbf{P}_{\nu}-a.e. \; \omega\in \Omega.
\end{equation}

Given $\varepsilon>0$ and $\rho_{\varepsilon}\in \mathcal{H}$, adapting the proof of Theorem 13.1 in \cite{Mane1} we can take a countable  partition $\mathcal{P}$ such that diam$\mathcal{P}_{\omega}(x)\le \rho_{\varepsilon}(\omega,x)$ for $\mathbf{P}_{\nu}\times\mu$-a.e. $(\omega,x)$. Moreover,
since $\mu$ is $\alpha$-invariant and is absolutely continuous with respect to the Lebesgue measure $m$,
\begin{equation}\label{mtomu1}
\limsup_{n\rightarrow\infty} -\frac{1}{n}\log\mu\big(\mathcal{P}_{\omega,n}(x)\big)=\limsup_{n\rightarrow\infty} -\frac{1}{n}\log m\big(\mathcal{P}_{\omega,n}(x)\big), \;\;\mathbf{P}_{\nu}\times\mu-a.e. (\omega,x).
\end{equation}
Hence by (\ref{SMB}) and (\ref{mtomu1}),
\begin{eqnarray*}
h_{\mu}(f)\ge h_{\mu}(f,\mathcal{P})&=&\int_M \limsup_{n\rightarrow\infty} -\frac{1}{n}\log \mu\big(\mathcal{P}_{\omega,n}(x)\big)d\mu(x)\\
&=&\int_M \limsup_{n\rightarrow\infty} -\frac{1}{n}\log m\big(\mathcal{P}_{\omega,n}(x)\big)d\mu(x)\\
&\ge&\int_M \limsup_{n\rightarrow\infty} -\frac{1}{n}\log m(B_n(\omega, \rho_{\varepsilon},x))d \mu(x)
\end{eqnarray*}
for $\mathbf{P}_{\nu}-a.e. \; \omega\in \Omega$.

Since $\varepsilon>0$ is arbitrary, (\ref{hmu}) holds.
\end{proof}

In order to evaluate $h_{\mu}(f, \mathcal{Q}_p)$ that in Lemma \ref{powerrule} and $m(B_n(\omega, \rho_{\varepsilon},x))$ that in Lemma \ref{localentropy}, we use the Lyapunov charts of the system. The following properties about Lyapunov charts derive from \cite{Hu}.

 Recall that for each generator $f_i$ of $\alpha$, $\{(\lambda_j(x,f_i),d_j(x,f_i)): 1\le j\le r(x,f_i), x\in \Gamma_i\}$ is its spectrum, and $\{(\lambda_{i,j}(x),d_j(x)): 1\le j\le s(x), x\in \Gamma\}$ is the spectrum of $\alpha$. Given $x\in \Gamma$, for $1\le i\le k$ let
$$
\lambda_+(x,f_i)=\min_{1\le j\le r(x,f_i)} \{\lambda_j(x,f_i):\lambda_j(x,f_i)>0\},
$$
$$
\lambda_-(x,f_i)=\max_{1\le j\le r(x,f_i)} \{\lambda_j(x,f_i):\lambda_j(x,f_i)<0\}
$$
and
$$
\Delta\lambda(x,f_i)=\min_{1\le j\le r(x,f_i)-1}\{\lambda_{j+1}(x,f_i)-\lambda_j(x,f_i)\}.
$$
Take
\begin{equation}\label{gamma}
0<\gamma(x)\le\frac{1}{100d}\min\{\lambda_+(x,f_i),-\lambda_-(x,f_i),\Delta\lambda(x,f_i):1\le i\le k\}.
\end{equation}
Clearly, the function $\gamma:\Gamma\longrightarrow (0,+\infty)$ defined as above can be chosen to be measurable and be invariant under $\alpha$, i.e., $\gamma(f_i(x))=\gamma(x)$ for any $1\le i\le k$.

 Let $\langle\cdot,\cdot\rangle$ and $|\cdot,\cdot|$ denote the usual inner product and norm in $\mathbb{R}^d$ respectively. Also, for $r>0$, let $\tilde{B}(r)$ be the ball in $\mathbb{R}^d$ centered at origin of radius $r$.

\begin{proposition}\label{Lyapunovcharts}
For a function $\gamma$
 defined as above, there exists a measurable function $\ell:\Gamma\longrightarrow [1,+\infty)$
 with $\ell(f_i^{\pm1}(x))\le \ell(x)e^{\gamma(x)}$, and a set of embeddings $\Phi_x : \tilde{B} (\ell(x)^{-1}) \longrightarrow M$
at each point $x\in \Gamma$ such that the following properties hold.

(1)  $\Phi_x(0)=x$, and the preimages $R_j(x)=D\Phi_x(0)^{-1}F_j(x)$ are mutually
orthogonal in $\mathbb{R}^d$, where $1\le j\le s(x)$.

(2) Let $\tilde{f}_{i,x} =\Phi_{f_i(x)}^{-1}\circ f_i \circ \Phi_x$ be the connecting map between the charts at $x$ and $f_i(x)$. For any $1\le j\le s(x)$ and $u\in R_j(x)$,
\begin{equation}\label{hyperbolicity}
|u|e^{\lambda_{i,j}(x)-\gamma(x)}\le|D\tilde{f}_{i,x}(0)u|\le |u|e^{\lambda_{i,j}(x)+\gamma(x)}.
\end{equation}

(3) Let $L(g)$ be the Lipschitz constant of a function $g$, then for each $1\le i\le k$,
$$
L(D(\tilde f_{i,x})), L(D(\tilde f_{i,x}^{-1}))\le \ell(x)
$$
and
$$
L\bigl(\tilde f_{i,x} - D\tilde f_{i,x}(0)\bigr), \
L\bigl(\tilde f_{i,x}^{-1} - D\tilde f_{i,x}^{-1}(0)\bigr) \le e^{\gamma(x)}
$$
on the corresponding domains.

(4) There exists $\lambda(x)$ depending on $\gamma(x)$ and the exponents (hence $\lambda$ is invariant under $\alpha$, i.e.,
$\lambda(f_i(x))=\lambda(x)$ for any $1\le i\le k$)  such that for each $i$,
$$
|\tilde{f}_{i,x}^{\pm1}u|\le e^{\lambda(x)}|u|, \;\;\;\;\forall u\in \tilde{B}( e^{-\lambda(x)-\gamma(x)}\ell(x)^{-1}),
$$
and hence for any $0<\varepsilon<e^{-\gamma(x)}$,
\begin{equation}\label{maptochart}
\tilde{f}_{i,x}\tilde{B}( \varepsilon e^{-\lambda(x)}\ell(x)^{-1})\subset \tilde{B}( \varepsilon \ell(f_ix)^{-1}).
\end{equation}

(5) For all $u, v \in \tilde{B}(\ell(x)^{-1})$, we have
\begin{equation}\label{metricsrelation}
K^{-1}d(\Phi_x(u), \Phi_x(v))\le |u-v|\le \ell(x)d(\Phi_x(u), \Phi_x(v)),
\end{equation}
for some universal constant $K\geq 2d$.
\end{proposition}

We shall refer to any system of local charts $\{\Phi_x, x\in \Gamma\}$ satisfying (1)-(5) as
$(\gamma,\ell)$-\emph{charts} for $\alpha$. For any constant $l>0$, denote $\Gamma_l=\{x\in M: \ell(x)\le l\}$. By the above proposition, for each $x\in \Gamma_l$ and any $0<\varepsilon<1$, we have
\begin{equation}\label{ballsrelation}
\Phi_x^{-1}B\big(x,\frac{\varepsilon}{l}\big)\subset \tilde{B}(\varepsilon),\;\;\text{and}\;\;\Phi_x\tilde{B}\big(\frac{\varepsilon}{K}\big)\subset B(x,\varepsilon).
\end{equation}

Clearly,
$$
\Gamma_l\subset \Gamma_{l+1}\;\;\text{and}\;\;\Gamma=\bigcup_{l\ge 1}\Gamma_l.
$$
 Such $\Gamma_l$'s are called \emph{Pesin sets} of $\alpha$.

\begin{proof}[Proof of Theorem 1]

Let $\Psi:\Omega\times M\longrightarrow \Omega\times M$ be the induced skew product transformation as in (\ref{Psi}). Since $\Gamma$ is $\alpha$-invariant  and with full measure, we will always concentrate on the restriction of  $\Psi$ to $\Omega\times \Gamma$.

\emph{Step 1.}  We first show the inequality
\begin{equation}\label{entropyformula1(1)}
h_{\mu}(f)\leq\int_M\max_{J\subset \{1,\cdots,s(x)\}}\sum_{j\in J}\sum_{i=1}^k \nu_i d_j(x)\lambda_{i,j}(x)d\mu(x).
\end{equation}

We introduce the notations $\tilde{f}_{\omega,x}=\Phi_{f_{\omega}x}^{-1}\circ f_{\omega}\circ\Phi_x$ and
$\tilde{f}^i_{\omega,x}=\tilde{f}_{\sigma^{i-1}\omega,x}\circ\cdots\circ\tilde{f}_{\sigma\omega,x}\circ\tilde{f}_{\omega,x}$ for $(\omega, x)\in \Omega\times \Gamma$ and $i\in \mathbb{N}$.

Fix $\varepsilon>0$ and $n\in \mathbb{N}$ arbitrarily. Let $p\in \mathbb{N}$. We write $\omega\in \Omega_p$ if for each $(\omega, x)\in \Omega\times \Gamma$ and $0\leq i\leq n$ we have
\begin{equation}\label{smallball2}
f_{\omega}^iB(y, \frac{2\varepsilon}{p})\subset B(f_{\omega}^i(x), \frac{1}{\ell(f_{\omega}^i(x))^2})
\end{equation}
and
\begin{equation}\label{smallball3}
\frac{1}{2}D\tilde{f}_{\omega,x}^i(0)\Phi_x^{-1}(B(y, \frac{\varepsilon}{p}))\subset \Phi_{f_{\omega}^i(x)}^{-1}B(f_{\omega}^i(B(y, \frac{\varepsilon}{p})))\subset 2D\tilde{f}_{\omega,x}^i(0)\Phi_x^{-1}(B(y, \frac{\varepsilon}{p}))
\end{equation}
for any $\displaystyle y\in B(x, \frac{\varepsilon}{p})$.
One can check that each $\Omega_p$ is measurable and for every $\omega$ there exists $p\in \mathbb{N}$ such that $\omega\in \Omega_p$. Consequently, $\lim_{p\rightarrow +\infty}\mathbf{P}_{\nu}(\Omega_p)=1$ since $\Omega_p\subset \Omega_{p+1}$ for all $p\in \mathbb{N}$.

For each $p\in \mathbb{N}$, take a maximal $\displaystyle\frac{\varepsilon}{p}$-seperated set $E_p$ of $M$. We then define
a measurable partition $\mathcal{Q}_p = \{\mathcal{Q}_p(x) : x \in E_p\}$ of $M$ such that $\mathcal{Q}_p(x)\subset \overline{\text{int}(\mathcal{Q}_p(x))}$ and
$$
 \text{int}(\mathcal{Q}_p(x))
= \{y\in M: d(y, x)< d(y, x_i) \text{ if } x\neq x_i\in E_p\}
$$
for every $x\in E_p$.
Clearly,
\begin{equation}\label{twoballs}
 B(x, \frac{\varepsilon}{2p})\subset\mathcal{Q}_p(x)\subset B(x, \frac{\varepsilon}{p})
\end{equation}
  for all $x\in E_p$.
Then by Lemma \ref{powerrule},
\begin{equation}\label{mixing}
nh_{\mu}(f)=h_{\mu}(f^n)=\lim_{p\rightarrow +\infty}h_{\mu}(f^n, \mathcal{Q}_p).
\end{equation}

Note that for each finite partition $\mathcal{Q}$ of $M$ one has¡¡
\begin{eqnarray*}
h_{\mu}(f, \mathcal{Q})&=&\lim_{t\longrightarrow +\infty}\frac{1}{t}\int_{\Omega}
H_{\mu}
\big(\bigvee_{i=0}^{t-1}(f_{\omega}^i)^{-1}\mathcal{\mathcal{Q}}
\big)\;d\mathbf{P}_{\nu}(\omega)\\
&\leq&\lim_{t\longrightarrow +\infty}\frac{1}{t}\int_{\Omega}
\Big[\sum_{i=1}^{t-1}H_{\mu}((f_{\omega}^i)^{-1}\mathcal{Q}|(f_{\omega}^{i-1})^{-1}\mathcal{Q})+
H_{\mu}(\mathcal{Q})\Big]\;d\mathbf{P}_{\nu}(\omega)\\
&=&\lim_{t\longrightarrow +\infty}\frac{1}{t}\Big[\int_{\Omega}
(t-1)H_{\mu}((f_{\omega})^{-1}\mathcal{Q}|\mathcal{Q})+H_{\mu}(\mathcal{Q})\;d\mathbf{P}_{\nu}(\omega)\Big]\\
&=&\int_{\Omega}H_{\mu}((f_{\omega})^{-1}\mathcal{Q}|\mathcal{Q})d\mathbf{P}_{\nu}(\omega).
\end{eqnarray*}
Thus from (\ref{mixing}), it follows that
\begin{eqnarray}\label{twoparts1}
nh_{\mu}(f)
&\leq&\limsup_{p\rightarrow +\infty}\int_{\Omega}H_{\mu}((f^n_{\omega})^{-1}\mathcal{Q}_p|\mathcal{Q}_p)d\mathbf{P}_{\nu}(\omega)\notag\\
&\leq&\limsup_{p\rightarrow +\infty}\int_{\Omega_p}H_{\mu}((f^n_{\omega})^{-1}\mathcal{Q}_p|\mathcal{Q}_p)d\mathbf{P}_{\nu}(\omega)\\
&+&\limsup_{p\rightarrow +\infty}\int_{\Omega\setminus\Omega_p}H_{\mu}((f^n_{\omega})^{-1}\mathcal{Q}_p|\mathcal{Q}_p)
d\mathbf{P}_{\nu}(\omega)\notag.
\end{eqnarray}

Note that for any $p\in \mathbb{N}$ and $\omega\in \Omega$,
\begin{eqnarray*}
H_{\mu}((f^n_{\omega})^{-1}\mathcal{Q}_p|\mathcal{Q}_p)
&=&H_{\mu}(\mathcal{Q}_p|f^n_{\omega}(\mathcal{Q}_p))\\
&=&\sum_{D\in f^n_{\omega}(\mathcal{Q}_p)}H_{\mu}(\mathcal{Q}_p| D)\mu(D)\\
&\leq& \sum_{D\in f^n_{\omega}(\mathcal{Q}_p)}\log \big(\text{card}\{C\in \mathcal{Q}_p: C\cap D\neq\emptyset\}\big).
\end{eqnarray*}
We first give a uniform exponential estimate for $\text{card}\{C\in \mathcal{Q}_p: C\cap D\neq\emptyset\}$ and then we will get a finer exponential bound for $\omega\in \Omega_p$. We will apply the approach in the proof of Theorem S.2.13 of \cite{Katok} to get these estimates.

\emph{Claim 1.} There exists a positive constant $K_1$ which is independent of $\omega$ such that for any $D\in f^n_{\omega}(\mathcal{Q}_p)$,
\begin{equation}\label{card1}
\text{card}\{C\in \mathcal{Q}_p: C\cap D\neq\emptyset\}\leq K_1\big(\max_{1\leq i\leq k}\sup_{x\in M}\|Df_i(x)\|\big)^{nd}.
\end{equation}
\emph{Proof of Claim 1}. By the Mean Value Theorem, for any $A\subset M$, we have
\begin{eqnarray*}
\text{diam}(f_{\omega}^n(A))&\leq& \text{diam }A \cdot \sup_{x\in M}\|Df_{\omega}^n(x)\|\\
&\leq& \text{diam }A\cdot \big(\max_{1\leq i\leq k}\sup_{x\in M}\|Df_i(x)\|\big)^n.
\end{eqnarray*}
Thus if $C\cap D\neq \emptyset$ then $C$ is contained in the $\displaystyle\frac{2\varepsilon}{p}$-neighborhood of $D$ whose diameter is at most $\big(\big(\displaystyle\max_{1\leq i\leq k}\sup_{x\in M}\|Df_i(x)\|\big)^n+2\big)\cdot \frac{2\varepsilon}{p}$ and whose volume is hence bounded by
$$
const.\big(\big(\big(\displaystyle\max_{1\leq i\leq k}\sup_{x\in M}\|Df_i(x)\|\big)^n+2\big)\cdot \frac{2\varepsilon}{p}\big)^d.
$$
On the other hand $C$ contains a ball of radius $\displaystyle\frac{\varepsilon}{2p}$ by (\ref{twoballs}) and hence the volume of $C$ is at least $const. \big(\displaystyle\frac{\varepsilon}{p}\big)^d$. Comparing these estimates gives (\ref{card1}) and hence claim 1 holds.

\emph{Claim 2.}  There exists a positive constant $K_2$ which is independent of $\omega\in \Omega_p$ such that for any $D\in f^n_{\omega}(\mathcal{Q}_p)$ and each $x\in \Gamma\cap\big(f_{\omega}^n\big)^{-1}D$,
\begin{equation}\label{card2}
\text{card}\{C\in \mathcal{Q}_p: C\cap D\neq\emptyset\}\leq K_2\prod_{j=1}^{s(x)}\max_{0\leq t\leq n-1}\prod_{r=0}^te^{d_j(x)(\lambda_{\sigma^r\omega ,j}(x)+\gamma(x))},
\end{equation}
in which $\lambda_{\sigma^r\omega, j}(x)=\lambda_{i,j}(x)$ if $f_{\sigma^r\omega}=f_i$.

\emph{Proof of Claim 2}.
Let $C'=\big(f_{\omega}^n\big)^{-1}D$. Pick some point $x\in \Gamma\cap C'$ and let $B=B(x, 2\text{diam }C')$. Then, by (\ref{smallball2}), $D\subset B_0:=\exp_{f_{\omega}^n(x)}\big(Df_{\omega}^n(x)(\exp_x^{-1}(B))\big).$
If $C\in \mathcal{Q}_p, B_0\cap C\neq\emptyset$, then $C\subset B_1:=\{y:d(y,B_0)<\text{diam }\mathcal{Q}_p\}.$
So by (\ref{hyperbolicity}), (\ref{metricsrelation}), (\ref{smallball2}) and (\ref{smallball3}),
\begin{eqnarray*}
&&\text{card }\{C\in \mathcal{Q}_p:D\cap C\neq\emptyset\}\leq \frac{\text{vol}(B_1)}{(\text{diam }\mathcal{Q}_p)^d}\\
&\leq&\frac{K^d\cdot\text{vol}\big(\Phi_{f_{\omega}^n(x)}^{-1} B_1\big)}{(\frac{\varepsilon}{p})^d}
\leq\frac{K^d\cdot\text{vol}\big(\tilde{f}_{\omega,x}^n(\Phi_x^{-1} B)\big)}{(\frac{\varepsilon}{p})^d}\\
&\leq&\frac{const.\cdot K^d\cdot  (\frac{8\varepsilon}{p})^d \displaystyle\prod_{j=1}^{s(x)}\max_{0\leq t\leq n-1}\prod_{r=0}^te^{d_j(x)(\lambda_{\sigma^r\omega ,j}(x)+\gamma(x))}}{(\frac{\varepsilon}{p})^d}.
\end{eqnarray*}
Hence (\ref{card2}) and then Claim 2 hold.

Since $\lim_{p\rightarrow +\infty}\mathbf{P}_{\nu}(\Omega_p)=1$, applying (\ref{card1}) that in Claim 1 and (\ref{card2}) that in Claim 2 to (\ref{twoparts1}), we have
\begin{eqnarray*}
h_{\mu}(f)&\leq&\limsup_{n\rightarrow\infty}\frac{1}{n}\int_{\Omega}\int_M\log\Big(\prod_{j=1}^{s(x)}\max_{0\leq t\leq n-1}\prod_{r=0}^te^{d_j(x)(\lambda_{\sigma^r\omega ,j}(x)+\gamma(x))}\Big)d\mu(x) d\mathbf{P}_{\nu}(\omega)\notag\\
&\leq&\int_M\int_{\Omega}\sum_{j=1}^{s(x)}\limsup_{n\rightarrow\infty}\frac{1}{n}\max_{0\leq t\leq n-1}\sum_{r=0}^td_j(x)(\lambda_{\sigma^r\omega, j}(x)+\gamma(x))d\mathbf{P}_{\nu}(\omega)d\mu(x).
\end{eqnarray*}
By (\ref{gamma}), $\gamma(x)$ can be taken arbitrarily small, so we have
 \begin{eqnarray}\label{step1}
h_{\mu}(f)\leq\int_M\int_{\Omega}\sum_{j=1}^{s(x)}\limsup_{n\rightarrow\infty}\frac{1}{n}\max_{0\leq t\leq n-1}\sum_{r=0}^td_j(x)\lambda_{\sigma^r\omega, j}(x)d\mathbf{P}_{\nu}(\omega)d\mu(x).
\end{eqnarray}
 Now for any $x\in \Gamma$ and each $1\leq j\leq s(x)$, define a function
\begin{equation}\label{eta_j1}
\eta^{(j)}_x:\Omega\longrightarrow \mathbb{R}^+,\;
\eta^{(j)}_x(\omega)=d_j(x)\lambda_{\omega, j}(x).
\end{equation}
 Clearly, it is $\mathbf{P}_{\nu}$-integrable.
Since $\mathbf{P}_{\nu}$ is $\sigma$-ergodic, by Birkhoff's Ergodic Theorem, we have that for $\mathbf{P}_{\nu}$-almost all $\omega\in \Omega$,
$$
\lim_{n\rightarrow\infty}\frac{1}{n}\sum_{r=0}^{n-1}\eta^{(j)}_x(\sigma^r\omega)
=\int_{\Omega}\eta^{(j)}_x(\omega)d\mathbf{P}_{\nu}(\omega),
$$
and hence,
$$
\lim_{n\rightarrow\infty}\frac{1}{n}\sum_{r=0}^{n-1}d_j(x)\lambda_{\sigma^r\omega, j}(x)
=\int_{\Omega}d_j(x)\lambda_{\omega ,j}(x)d\mathbf{P}_{\nu}(\omega)=\sum_{i=1}^k\nu_id_j(x)\lambda_{i, j}(x).
$$
By Lemma 9.2 of \cite{Einsiedler}, for any sequence of real numbers $\{a_n\}$ such that $\displaystyle\frac{a_n}{n}\rightarrow a$ as $n\rightarrow\infty$, we have $\displaystyle\max\limits_{1\leq t\leq n}\frac{a_t}{n}\rightarrow \max\{a,0\}$ as $n\rightarrow\infty$. Then let $a_n=\sum_{r=0}^{n-1}\eta^{(j)}_x(\sigma^r\omega)$, we have
\begin{eqnarray}\label{lim1}
\lim_{n\rightarrow\infty}\frac{1}{n}\max_{0\leq t\leq n-1}\sum_{r=0}^td_j(x)\lambda_{\sigma^r\omega, j}(x)
=\max\bigl\{\sum_{i=1}^k\nu_id_j(x)\lambda_{i, j}(x), \;0\bigr\}.
\end{eqnarray}
Therefore,  the inequality in (\ref{step1}) becomes
$$
h_{\mu}(f)\leq\int_M\max_{J\subset \{1,\cdots,s(x)\}}\sum_{j\in J}\sum_{i=1}^k \nu_i d_j(x)\lambda_{i,j}(x)d\mu(x),
$$
i.e., the inequality (\ref{entropyformula1(1)}) holds.

\emph{Step 2.} We show the reverse inequality
\begin{equation}\label{entropyformula1(2)}
h_{\mu}(f)\geq\int_M\max_{J\subset \{1,\cdots,s(x)\}}\sum_{j\in J}\sum_{i=1}^k \nu_i d_j(x)\lambda_{i,j}(x)d\mu(x).
\end{equation}

From Lemma \ref{localentropy}, the main work is to calculate the reduction rate of the volume of $B_n(\omega, \rho_{\varepsilon}, x)$ for some suitable $\rho_{\varepsilon}$ via the Lyapunov exponents of the generators of the $\mathbb{Z}^k$-action $\alpha$.

Given $\varepsilon>0$, take $l=l(\varepsilon)>1$ such that $\displaystyle \mu(\Gamma_l)>1-\frac{\varepsilon}{2}$. By Poincar$\acute{e}$'s Recurrence Theorem, almost all points in $\Omega\times \Gamma_l$ return infinitely often to  $\Omega\times \Gamma_l$ under positive iteration by $\Psi$.
Let $\tau$ be the return function of
$\Omega\times \Gamma_l$, that is, let $\tau(\omega,x)$ be the least positive integer such that $\Psi^{\tau(\omega,x)}(\omega,x)\in \Omega\times \Gamma_l$. In another word, for any $\omega\in \Omega$,  $\tau(\omega,x)$ is the least positive integer such that $f_{\omega}^{\tau(\omega,x)}(x)\in \Gamma_l$. Clearly, $\tau$ is integrable.  Extend $\tau$ to $\Omega\times M$ by putting $\tau(\omega,x)=0$ for $(\omega,x)\notin \Omega\times \Gamma_l$. For the above $\varepsilon$, define a function $\rho_{\varepsilon}:\Omega\times M\longrightarrow (0,\varepsilon)$ by
\begin{equation}\label{radiusfunction}
\rho_{\varepsilon}(\omega,x)=
\frac{\varepsilon}{l^2}e^{-(\lambda(\omega,x)+\gamma(\omega,x))\tau(\omega,x)},
\end{equation}
where
$$
\left\{\begin{array}{ll}
\gamma(\omega,x)=\gamma(x)\;\;\text{and}\;\;\lambda(\omega,x)=\lambda(x) \quad &\;\;\text{if }x\in \Gamma\\
\gamma(\omega,x)=\lambda(\omega,x)=1 \quad &\;\;\text{if }x\notin \Gamma.
\end{array}\right.
$$
Recall that $\gamma(x)$ and $\lambda(x)$ are all integrable and $\alpha$-invariant on $\Gamma$, so the function (\ref{radiusfunction}) is reasonable and integrable.

 Let
$\chi_{\Omega\times (\Gamma-\Gamma_l)}$ be the characteristic function of $\Omega\times (\Gamma-\Gamma_l)$.
By Birkhoff's Ergodic Theorem, the function
\begin{eqnarray*}
\chi^*(\omega,x):&=&\lim_{n\longrightarrow +\infty}\frac{1}{n}\sum_{i=0}^{n-1}\chi_{\Omega\times (\Gamma-\Gamma_l)}\Psi^i(\omega,x)\\
&=&\lim_{n\longrightarrow +\infty}\frac{1}{n}\text{card}\{0\le i\le n-1: \Psi^i(\omega,x)\in \Omega\times (\Gamma-\Gamma_l)\}
\end{eqnarray*}
is defined for $\mathbf{P}_{\nu}\times\mu$-almost all $(\omega,x)\in \Omega\times \Gamma$. Moreover,
\begin{eqnarray*}
\varepsilon &\geq&\mathbf{P}_{\nu}\times\mu(\Omega\times (\Gamma-\Gamma_l))
=\int_{\Omega\times \Gamma} \chi_{\Omega\times (\Gamma-\Gamma_l)}d\mathbf{P}_{\nu}\times\mu=\int_{\Omega\times \Gamma} \chi^*d\mathbf{P}_{\nu}\times\mu\\
&\geq& \int_{\{(\omega,x)\in \Omega\times \Gamma:\chi^*(\omega,x)>\sqrt{\varepsilon}\}} \chi^*d\mathbf{P}_{\nu}\times\mu\\
&>& \sqrt{\varepsilon} \cdot \mathbf{P}_{\nu}\times\mu\big(\{(\omega,x)\in \Omega\times \Gamma:\chi^*(\omega,x)>\sqrt{\varepsilon}\}\big),
\end{eqnarray*}
and hence
$$
\mathbf{P}_{\nu}\times\mu\big(\{(\omega,x)\in \Omega\times \Gamma:\chi^*(\omega,x)\le\sqrt{\varepsilon}\}\big)\geq 1-\sqrt{\varepsilon}.
$$
By Egorov's Theorem, there exist $\Delta\subset\Omega\times \Gamma$ with $\displaystyle\mathbf{P}_{\nu}\times\mu(\Delta)>1-\frac{\varepsilon}{2}$ and $N_0>0$ such that  for each $(\omega,x)\in \Delta$ and any $n>N_0$,
\begin{equation}\label{frequency}
\text{card}\{0\le i\le n-1: \Psi^i(\omega,x)\in \Omega\times (\Gamma-\Gamma_l)\le 2n\sqrt{\varepsilon}.
\end{equation}
Clearly,
\begin{equation}\label{measure}
\mathbf{P}_{\nu}\times\mu\big(\Delta\cap(\Omega\times \Gamma_l)\big)>1-\varepsilon.
\end{equation}

Given $(\omega,x)\in \Delta\cap(\Omega\times \Gamma_l)$ and $n>N_0$.
By (\ref{maptochart}) and the definition of $\rho_{\varepsilon}$ in (\ref{radiusfunction}),
\begin{equation}\label{local1}
\Phi_{f_{\omega}^ix}^{-1}B\big(f_{\omega}^ix,\rho_{\varepsilon}(\sigma^i\omega,f_{\omega}^ix)\big)
\subset\tilde{B}\big(\ell(f_{\omega}^ix)^{-1}\varepsilon\big)
\end{equation}
for any $0\le i\le n-1$,
and
\begin{equation}\label{local2}
\Phi_{f_{\omega}^{i+1}x}^{-1}\Big(f_{\sigma^{i+1}\omega}B\big(f_{\omega}^ix,
\rho_{\varepsilon}(\sigma^i\omega,f_{\omega}^ix)\big)\Big)
\subset\tilde{B}\big(\ell(f_{\omega}^{i+1}x)^{-1}\varepsilon\big)
\end{equation}
for any $0\le i\le n-2$.
Let
$$
\tilde{B}_n^{(1)}(\omega, \rho_{\varepsilon}, x)=\big\{u\in \mathbb{R}^d:|\tilde{f}^i_{\omega,x}(u)|\leq \frac{1}{K}\rho_{\varepsilon}(\sigma^i\omega, f_{\omega}^i(x)), 0\le i\le n-1\big\}
$$
and
$$
\tilde{B}_n^{(2)}(\omega, \rho_{\varepsilon}, x)=\big\{u\in \mathbb{R}^d:|\tilde{f}^i_{\omega,x}(u)|\leq \ell(f_{\omega}^ix)\cdot\rho_{\varepsilon}(\sigma^i\omega, f_{\omega}^i(x)), 0\le i\le n-1\big\}.
$$
Hence by (\ref{local1}) and (\ref{local2}), (\ref{metricsrelation}) ensures that
\begin{equation}\label{ballsrelation1}
\Phi_{x}\tilde{B}_n^{(1)}(\omega, \rho_{\varepsilon}, x)\subset B_n(\omega, \rho_{\varepsilon}, x)\subset \Phi_{x}\tilde{B}_n^{(2)}(\omega, \rho_{\varepsilon}, x).
\end{equation}

In the following, we will estimate the reduction rate of the volumes of $\tilde{B}_n^{(1)}(\omega, \rho_{\varepsilon}, x)$ and $\tilde{B}_n^{(2)}(\omega, \rho_{\varepsilon}, x)$ for $(\omega,x)\in \Delta\cap(\Omega\times \Gamma_l)$ and $n>N_0$.
By (\ref{radiusfunction}), (\ref{frequency}) and (\ref{metricsrelation}),
\begin{equation}\label{smallball}
\tilde{B}\big(\frac{\varepsilon}{Kl^2}e^{-(\lambda(\omega,x)+\gamma(\omega,x))2n\sqrt{\varepsilon}}\big)\subset
\Phi_{f_{\omega}^ix}^{-1}B\big(f_{\omega}^ix,\rho_{\varepsilon}(\sigma^i\omega,f_{\omega}^ix)\big),\;0\le i\le n-1,
\end{equation}
and by (\ref{local1}),
\begin{equation}\label{bigball}
\Phi_{f_{\omega}^ix}^{-1}B\big(f_{\omega}^ix,\rho_{\varepsilon}(\sigma^i\omega,f_{\omega}^ix)\big)
\subset\tilde{B}\big(\ell(f_{\omega}^ix)^{-1}\varepsilon\big) \subset\tilde{B}\big(\varepsilon\big)\;0\le i\le n-1.
\end{equation}
Reducing $\varepsilon$ if necessary such that for any $0<\eta\le\varepsilon$, we have that
\begin{equation}\label{linearappro}
D\tilde{f}^{\pm1}_{i,x}(0)\tilde{B}(a\eta)\subset \tilde{f}^{\pm1}_{i,x}\tilde{B}(\eta)
\subset D\tilde{f}^{\pm1}_{i,x}(0)\tilde{B}(A\eta),\;\;1\le i\le k, x\in \Gamma_l,
\end{equation}
hold for two positive constants $a$ and $A$ whenever $\tilde{f}_{i,x}\tilde{B}(\eta)$ makes sense (recall that $\tilde{f}_{i,x}$ is only well defined on $\tilde{B}(\ell(x)^{-1})$).
For any $1\le j\le s(x)$, denote
$$
\tilde{B}_{n,j}^{(1)}(\omega, \rho_{\varepsilon}, x)=\big\{u\in R_j:|D\tilde{f}^i_{\omega,x}(0)u|\leq \frac{a\varepsilon}{Kl^2}e^{-(\lambda(\omega,x)+\gamma(\omega,x))2n\sqrt{\varepsilon}}, 0\le i\le n-1\big\}
$$
and
$$
\tilde{B}_{n,j}^{(2)}(\omega, \rho_{\varepsilon}, x)=\big\{u\in R_j:|D\tilde{f}^i_{\omega,x}(0)u|\leq A\varepsilon, 0\le i\le n-1\big\}.
$$
By Proposition \ref{Lyapunovcharts}, $R_j(x), 1\le j\le s(x)$, are mutually
orthogonal in $\mathbb{R}^d$. Therefore there exist two positive numbers $b$ and $B$ only depending on the dimension $d$ such that
\begin{equation}\label{productvolume}
b\prod_{j=1}^{s(x)}\text{vol}\tilde{B}_{n,j}^{(i)}(\omega, \rho_{\varepsilon}, x)\le\text{vol}\tilde{B}_n^{(i)}(\omega, \rho_{\varepsilon}, x)\le B\prod_{j=1}^{s(x)}\text{vol}\tilde{B}_{n,j}^{(i)}(\omega, \rho_{\varepsilon}, x),\;i=1,2.
\end{equation}
Now we estimate the  volumes of $\tilde{B}_{n,j}^{(1)}(\omega, \rho_{\varepsilon}, x)$ and $\tilde{B}_{n,j}^{(2)}(\omega, \rho_{\varepsilon}, x)$ for $1\le j\le s(x)$.
By (\ref{hyperbolicity}),
\begin{equation}\label{volumej}
\text{vol}\tilde{B}_{n,j}^{(1)}(\omega, \rho_{\varepsilon}, x)\ge b \Big(\frac{a\varepsilon}{Kl^2}e^{-(\lambda(\omega,x)+\gamma(\omega,x))2n\sqrt{\varepsilon}}\Big)^{d_j(x)}\min_{0\leq t\leq n-1}\prod_{r=0}^t\bigl(e^{\lambda_{\sigma^r\omega, j}(x)+\gamma(x)}\bigr)^{-d_j(x)},
\end{equation}
in which $\lambda_{\sigma^r\omega, j}(x)=\lambda_{p,j}(x)$ if $f_{\sigma^r\omega}=f_p$.
Hence
\begin{eqnarray}\label{Dlj}
&&\limsup_{n\rightarrow\infty}\bigl[
-\frac{1}{n}\log \text{vol}\tilde{B}_{n,j}^{(1)}(\omega, \rho_{\varepsilon}, x)\bigr]\notag\\
&\le& (\lambda(\omega,x)+\gamma(\omega,x))2 d_j(x) \sqrt{\varepsilon}+
\limsup_{n\rightarrow\infty}\frac{1}{n}\max_{0\leq t\leq n-1}\log\prod_{r=0}^te^{d_j(x)(\lambda_{\sigma^r\omega ,j}(x)+\gamma(x))}\\
&=&(\lambda(x)+\gamma(x))2 d_j(x) \sqrt{\varepsilon} +\limsup_{n\rightarrow\infty}\frac{1}{n}\max_{0\leq t\leq n-1}\sum_{r=0}^td_j(x)(\lambda_{\sigma^r\omega, j}(x)+\gamma(x)).\notag
\end{eqnarray}
As we have done for
$$
\limsup_{n\rightarrow\infty}\frac{1}{n}\max_{0\leq t\leq n-1}\sum_{r=0}^td_j(x)\lambda_{\sigma^r\omega, j}(x)
$$
in (\ref{step1}), we get an equality
\begin{eqnarray}\label{lim}
\lim_{n\rightarrow\infty}\frac{1}{n}\Bigl(\max_{0\leq t\leq n-1}\sum_{r=0}^td_j(x)(\lambda_{\sigma^r\omega, j}(x)+\gamma(x))\Bigr)
=\max\bigl\{\sum_{i=1}^k\nu_id_j(x)(\lambda_{i, j}(x)+\gamma(x)), \;0\bigr\}
\end{eqnarray}
for $\mathbf{P}_{\nu}$-almost all $\omega\in \Omega$, which is similar to that in (\ref{lim1}).
Therefore for $\mathbf{P}_{\nu}$-almost all $\omega\in \Omega$ and any $x\in \Gamma_l$ with $(\omega,x)\in \Delta\cap(\Omega\times \Gamma_l)$,
\begin{eqnarray}\label{B1}
&&\limsup_{n\rightarrow\infty}\bigl[
-\frac{1}{n}\log \text{vol}\tilde{B}_n^{(1)}(\omega, \rho_{\varepsilon}, x)\bigr]\notag\\
&\le&\limsup_{n\rightarrow\infty}\Bigl[-
\frac{1}{n}\log b\prod_{j=1}^{s(x)}\text{vol}\tilde{B}_{n,j}^{(1)}(\omega, \rho_{\varepsilon}, x)\Bigr](\text{by}\;(\ref{productvolume}))\notag\\
&\le&\sum_{j=1}^{s(x)}\Big((\lambda(x)+\gamma(x))2 d_j(x) \sqrt{\varepsilon} +\limsup_{n\rightarrow\infty}\frac{1}{n}\Bigl[\max_{0\leq t\leq n-1}\sum_{r=0}^td_j(x)(\lambda_{\sigma^r\omega, j}(x)+\gamma(x))\Bigr]\Big)\;\;(\text{by}\;(\ref{Dlj}))\notag\\
&=&\sum_{j=1}^{s(x)}\Big((\lambda(x)+\gamma(x))2 d_j(x) \sqrt{\varepsilon} +\max\bigl\{\sum_{i=1}^k\nu_id_j(x)(\lambda_{i, j}(x)+\gamma(x)), \;0\bigr\}\Big)\;\;(\mbox{by}\;(\ref{lim}))\notag\\
&=&\sum_{j=1}^{s(x)}(\lambda(x)+\gamma(x))2 d_j(x) \sqrt{\varepsilon} +\max_{J\subset \{1,\cdots,s(x)\}}\sum_{j\in J}\sum_{i=1}^k \nu_i d_j(x)(\lambda_{i, j}(x)+\gamma(x)).
\end{eqnarray}
In a similar manner, we can show that
\begin{eqnarray}\label{B2}
&&\limsup_{n\rightarrow\infty}\bigl[
-\frac{1}{n}\log \text{vol}\tilde{B}_n^{(2)}(\omega, \rho_{\varepsilon}, x)\bigr]\notag\\
&\ge&\limsup_{n\rightarrow\infty}\Bigl[-
\frac{1}{n}\log B\prod_{j=1}^{s(x)}\text{vol}\tilde{B}_{n,j}^{(2)}(\omega, \rho_{\varepsilon}, x)\Bigr](\text{by}\;(\ref{productvolume}))\notag\\
&\ge&\limsup_{n\rightarrow\infty}\frac{1}{n}\sum_{j=1}^{s(x)}\max_{0\leq t\leq n-1}\log\prod_{r=0}^te^{d_j(x)(\lambda_{\sigma^r\omega ,j}(x)-\gamma(x))}\;\;(\text{by}\;(\ref{hyperbolicity}))\notag\\
&=&\max_{J\subset \{1,\cdots,s(x)\}}\sum_{j\in J}\sum_{i=1}^k \nu_i d_j(x)(\lambda_{i, j}(x)-\gamma(x)).
\end{eqnarray}

By (\ref{B1}), (\ref{B2}), (\ref{ballsrelation1}) and (\ref{metricsrelation}), for $\mathbf{P}_{\nu}$-almost all $\omega\in \Omega$ and any $x\in \Gamma_l$ with $(\omega,x)\in \Delta\cap(\Omega\times \Gamma_l)$,
\begin{eqnarray}\label{maininequality}
&&\max_{J\subset \{1,\cdots,s(x)\}}\sum_{j\in J}\sum_{i=1}^k \nu_i d_j(x)(\lambda_{i, j}(x)-\gamma(x))\notag\\
&\le&\limsup_{n\rightarrow\infty}-\frac{1}{n}\log m(B_n(\omega, \rho_{\varepsilon}, x))\\
&\le&\sum_{j=1}^{s(x)}(\lambda(x)+\gamma(x))2 d_j(x) \sqrt{\varepsilon} +\max_{J\subset \{1,\cdots,s(x)\}}\sum_{j\in J}\sum_{i=1}^k \nu_i d_j(x)(\lambda_{i, j}(x)+\gamma(x)).\notag
\end{eqnarray}
Notice that $\mu(\Gamma_l)\longrightarrow 1$ and hence $\mathbf{P}_{\nu}\times\mu\big(\Delta\cap(\Omega\times \Gamma_l)\big)\longrightarrow 1$ as $\varepsilon\longrightarrow 0$. So for $\mathbf{P}_{\nu}$-a.e. $\omega\in \Omega$ and
$\mu$-a.e. $x\in\Gamma$,
 \begin{eqnarray}\label{maininequality1}
&&\max_{J\subset \{1,\cdots,s(x)\}}\sum_{j\in J}\sum_{i=1}^k \nu_i d_j(x)(\lambda_{i, j}(x)-\gamma(x))\notag\\
&\le&\lim_{\varepsilon\longrightarrow 0}\limsup_{n\rightarrow\infty}-\frac{1}{n}\log m(B_n(\omega, \rho_{\varepsilon}, x))\notag\\
&\le&\max_{J\subset \{1,\cdots,s(x)\}}\sum_{j\in J}\sum_{i=1}^k \nu_i d_j(x)(\lambda_{i, j}(x)+\gamma(x)).\notag
\end{eqnarray}
By (\ref{gamma}), $\gamma(x)$ can be taken arbitrarily small, so we have
 \begin{eqnarray}\label{maininequality2}
\lim_{\varepsilon\longrightarrow 0}\limsup_{n\rightarrow\infty}-\frac{1}{n}\log m(B_n(\omega, \rho_{\varepsilon}, x))=\max_{J\subset \{1,\cdots,s(x)\}}\sum_{j\in J}\sum_{i=1}^k \nu_i d_j(x)\lambda_{i, j}(x).
\end{eqnarray}
Note $\mu(\Gamma)=1$. By Lemma \ref{localentropy}, integrating both sides of the above equation we get the inequality (\ref{entropyformula1(2)}).

By the inequalities (\ref{entropyformula1(1)}) in Step 1 and (\ref{entropyformula1(2)}) in Step 2, we obtain the desired entropy formula (\ref{entropyformula1}).
\end{proof}

\section{Topological entropy for random $\mathbb{Z}^k$ and $\mathbb{Z}_+^k$-actions}

In the above section, we obtain an entropy formula for $C^2$ random $\mathbb{Z}^k$-actions on closed $C^{\infty}$ Riemannian manifolds via the Lyapunov exponents of the generators. By the variational principle (\ref{VP}), we obtain the lower bounds of topological entropy  immediately.

\begin{proposition}
For any  random ${\mathbb{Z}}^k$-action $f$ over $(\Omega, \mathcal{A}, \mathbf{P}_{\nu}, \sigma)$ as in Theorem 1 we have
\begin{equation}\label{entropyformula3}
h(f)\ge\int_M\max_{J\subset \{1,\cdots,s(x)\}}\sum_{j\in J}\sum_{i=1}^k \nu_i d_j(x)\lambda_{i,j}(x)d\mu(x).
\end{equation}
\end{proposition}

In the following we begin to consider the topological entropy for more general random $\mathbb{Z}_+^k$-actions generated by non-smooth or non-invertible maps.
 In the remaining of this section, we always assume that $\alpha$ is a $C^0$ ${\mathbb{Z}}_+^k$-action on a compact metric space $X$ with the generators $f_i, 1\leq i\leq k$, and $f$ is the induced random ${\mathbb{Z}}_+^k$-action over $(\Omega, \mathcal{A}, \mathbf{P}_{\nu}, \sigma)$, in which $\mathbf{P}_{\nu}=\nu^{\mathbb{Z}_+}$ for a Borel probability measure $\nu$  on $\mathcal{G}_{\alpha}$. Denote $\nu_i=\nu(f_i)$.

Let $b(\varepsilon)$ be the
minimum cardinality of covering of $X$ by $\varepsilon$-balls. Then
$$
D(X)=\limsup_{\varepsilon\rightarrow0}\frac{{\log}b(\varepsilon)}{|{\log}\varepsilon|}\in\mathbb{R}^+\cup\{\infty\}
$$
is called the \emph{ball dimension} of $X$.
It is well known that if a map $g:X\longrightarrow X$ is Lipschitzian with the Lipschitz constant $L(g)$, then
$$
h(g)\leq D(X)\log(\max\{1, L(g)\}),
$$
see Theorem 3.2.9 of \cite{Katok} for example. In the following, we give the corresponding inequalities for a  random ${\mathbb{Z}}_+^k$-action $f$.

\begin{proposition}\label{Lips}
 Let $f$ be a random ${\mathbb{Z}}_+^k$-action over $(\Omega, \mathcal{A}, \mathbf{P}_{\nu}, \sigma)$ induced by a $C^0$ ${\mathbb{Z}}_+^k$-action $\alpha$. If the generators $f_i, 1\leq i\leq k$, are all Lipschitz maps with the Lipschitz constants $L(f_i), 1\leq i\leq k$, respectively, then we have
\begin{equation}\label{Lipformula}
 h(f)\leq D(X)\sum_{i=1}^k \nu_i\log L^+(f_i),
\end{equation}
 where $L^+(f_i)=\max\{1, L(f_i)\}$.
\end{proposition}

\begin{proof}
From (\ref{topint}), $h(f)=\int h(f, \omega)\;d\mathbf{P}_{\nu}(\omega)$, where
$$
h(f, \omega)=\lim_{\varepsilon \rightarrow
0}\limsup_{n\rightarrow\infty} \frac{1}{n} \log r(\omega,n,\varepsilon,X).
$$
By a standard discussion, we can see that for any $\varepsilon>0, n>0$ and $\omega\in \Omega$,
$$
r(\omega,n,\varepsilon,X)\leq C \prod_{i=0}^{n-1}L^+(f_{\sigma^i\omega})^{D(X)},
$$
where $C$ is a constant which only depends on $\varepsilon$.
Hence
$$
h(f, \omega)\leq D(X)\limsup_{n\rightarrow\infty} \frac{1}{n}\sum_{i=0}^{n-1}\log L^+(f_{\sigma^i\omega}).
$$
Now we define a function
\begin{equation}\label{beta}
\beta:\Omega\longrightarrow \mathbb{R}^+,\;
\beta(\omega)=\log L^+(f_{\omega}).
\end{equation}
 Clearly, it is $\mathbf{P}_{\nu}$-integrable.
Since $\mathbf{P}_{\nu}$ is $\sigma$-ergodic, by Birkhoff's Ergodic Theorem, we have that for $\mathbf{P}_{\nu}$-almost all $\omega\in \Omega$,
$$
\lim_{n\rightarrow\infty}\frac{1}{n}\sum_{t=0}^{n-1}\beta(\sigma^t\omega)
=\int_{\Omega}\beta(\omega)d\mathbf{P}_{\nu}(\omega).
$$
Equivalently,
$$
\lim_{n\rightarrow\infty}\frac{1}{n}\sum_{t=0}^{n-1}\log L^+(f_{\sigma^t\omega})
=\sum_{i=1}^k \nu_i\log L^+(f_i).
$$
Therefore, the inequality (\ref{Lipformula}) holds.
\end{proof}

\begin{corollary}
Let $f$ be a random ${\mathbb{Z}}_+^k$-action over $(\Omega, \mathcal{A}, \mathbf{P}_{\nu}, \sigma)$ induced by a $C^1$ ${\mathbb{Z}}^k$-action $\alpha$ on a $d$-dimensional closed Riemannian
manifold $M$. Then
\begin{equation}\label{smoothformula}
 h(f)\leq d\sum_{i=1}^{k} \nu_i\log D^+(f_i),
\end{equation}
where $ D^+(f_i)=\max\{1,\sup_{ x\in M}\|Df_i(x)\|\}$.
\end{corollary}

When we define entropy we
only look at the future behavior of the considered system. In recent years,
Hurley \cite{Hurley}, Nitecki and Przytycki \cite{Nitecki}, Cheng and Newhouse \cite{Cheng} and Zhu \cite{Zhu07} formulated and studied several
entropy-like invariants based on the preimage structure of the systems in deterministic and random settings respectively. They all measure how complex the preimages spread the points on the phase space.  Some relations between entropy and these preimage entropies were investigated and hence the calculation of entropy becomes easier in certain cases. We only use the so-called ``pointwise  preimage entropy" to estimate the topological entropy for certain random ${\mathbb{Z}}_+^k$-actions.

Let $f$ be a random ${\mathbb{Z}}_+^k$-action over $(\Omega, \mathcal{A}, \mathbf{P}_{\nu}, \sigma)$ induced by a ${\mathbb{Z}}_+^k$-action $\alpha$ on $X$.  For any $\omega\in \Omega$,
let
$$
h_{(m)}(f, \omega)=\lim_{\varepsilon\rightarrow 0}\limsup
_{n\rightarrow \infty}\frac{1}
{n}\log\sup_{x\in X}r(\omega, n, \varepsilon, (f_{\omega}^n)^{-1}(x)).
$$
By a discussion similar to that in \cite{Kifer1}, we can see that $h_{(m)}(f,\omega)$ is $\mathbf{P}_{\nu}$-integrable. We call the quantity
\begin{equation}\label{hm}
h_{(m)}(f):=\int_{\Omega}h_{(m)}(f, \omega) d\mathbf{P}_{\nu}(\omega)
\end{equation}
the \emph{pointwise  preimage entropy} of $f$.

When $X$ is a finite graph and  the generators $f_i, 1\leq i\leq k$, are all continuous (and hence equi-continuous), then from a slight adaption of  Theorem 6.4  in \cite{Nitecki}, we have for any $\omega\in \Omega$,
\begin{equation}\label{h=hm}
h(f,\omega)=h_{(m)}(f,\omega).
\end{equation}
  When $X=M$ is a closed oriented Riemannian manifold and  the generators $f_i, 1\leq i\leq k$, are all expanding maps, i.e., $\max_{x\in M}\{\|Df_i(x)\|:x\in M,1\leq i\leq k\}>1$, then adapt Proposition 6.1 in \cite{Nitecki} to our case, we also have the equality (\ref{h=hm}) for each $\omega\in \Omega$.

  \begin{proposition}\label{twohi}
Let $f$ be a random ${\mathbb{Z}}_+^k$-action over $(\Omega, \mathcal{A}, \mathbf{P}_{\nu}, \sigma)$ induced by a ${\mathbb{Z}}_+^k$-action $\alpha$ on $X$.

(1) If $X$ is a finite graph and  the generators $f_i, 1\leq i\leq k$, are all homeomorphisms, then $h(f)=0$.

(2) If $X=M$ is a closed oriented Riemannian manifold and  the generators $f_i, 1\leq i\leq k$, are all expanding maps, then we have
\begin{equation}\label{expanding}
 h(f)=\sum_{i=1}^k \nu_i\log |\deg(f_i)|,
\end{equation}
 where $\deg(f_i)$ is the degree of $f_i$.
\end{proposition}

\begin{proof}
(1) Note that for any random ${\mathbb{Z}}_+^k$-action whose generators are all homeomorphisms, we have $h_{(m)}(f)=0$.

(2) Since all of the generators are expanding, for any $\omega\in \Omega$, we have
$$
h_{(m)}(f,\omega)\leq \limsup\limits_{n\rightarrow\infty}\frac{1}{n}
\log|\deg f_w^{n-1}|=\limsup\limits_{n\rightarrow\infty}\frac{1}{n}\sum_{i=0}^{n-1}
\log|\deg f_{\sigma^iw}|.
$$
Moreover, from Theorem 3.1 of \cite{Zhang}, we have
$$
h(f,\omega)\geq \limsup\limits_{n\rightarrow\infty}\frac{1}{n}\sum_{i=0}^{n-1}
\log|\deg f_{\sigma^iw}|.
$$
Combining the above two inequalities and (\ref{h=hm}), we have
\begin{equation}\label{equal}
h(f,\omega)=\limsup\limits_{n\rightarrow\infty}\frac{1}{n}\sum_{i=0}^{n-1}
\log|\deg f_{\sigma^iw}|.
\end{equation}

Now we define a function
\begin{equation}\label{zeta}
\zeta:\Omega\longrightarrow \mathbb{R}^+,\;
\zeta(\omega)=\log |\deg f_{w}|.
\end{equation}
Since $\mathbf{P}_{\nu}$ is $\sigma$-ergodic, by Birkhoff's Ergodic Theorem, we have that for $\mathbf{P}_{\nu}$-almost all $\omega\in \Omega$,
$$
\lim_{n\rightarrow\infty}\frac{1}{n}\sum_{t=0}^{n-1}\zeta(\sigma^t\omega)
=\int_{\Omega}\zeta(\omega)d\mathbf{P}_{\nu}(\omega).
$$
Equivalently,
$$
\lim_{n\rightarrow\infty}\frac{1}{n}\sum_{t=0}^{n-1}\log |\deg f_{\sigma^iw}|
=\sum_{i=1}^k \nu_i\log |\deg(f_i)|.
$$
Therefore, the equality (\ref{expanding}) holds.
\end{proof}

In (1) of the above proposition, we require the generators $f_i, 1\leq i\leq k,$ are all homeomorphisms. When we consider  the systems on special finite graphs: circle and interval, we can give more information on topological entropy of $f$ with non-invertible generators. Let $f$ be a random ${\mathbb{Z}}_+^k$-action $f$ over $(\Omega, \mathcal{A}, \mathbf{P}_{\nu}, \sigma)$. If $X=\mathbb{S}^1$ is the unit circle and the generators $f_i, 1\leq i\leq k,$ are all monotonic. From Theorem 2.1 of \cite{Zhu06}, for any $\omega\in \Omega$, we have
$$
h(f,\omega)=\limsup_{n\rightarrow\infty}
\frac{1}{n}\sum_ {i=0}^{n-1}\log|\deg(f_{\sigma^iw})|.
$$
If $X=I$ is an interval and the generators $f_i, 1\leq i\leq k,$ are all piecewise
monotonic, then it is not difficult to see that for any $\omega\in \Omega$,
$$
h(f,\omega)\leq\limsup_{n\rightarrow\infty}
\frac{1}{n}\sum_ {i=0}^{n-1}\log\mathcal {N}(f_{\sigma^iw}),
$$
where $\mathcal {N}(f_{\sigma^iw})$
denotes the number of intervals of monotonicity for $f_{\sigma^iw}$. Therefore, by the similar method we have applied to the proof of the above propositions, we have the following results immediately.
\begin{proposition}
Let $f$ be  a random ${\mathbb{Z}}_+^k$-action  over $(\Omega, \mathcal{A}, \mathbf{P}_{\nu}, \sigma)$ induced by a $C^0$ ${\mathbb{Z}}_+^k$-action $\alpha$ on $X$.

(1) If $X=\mathbb{S}^1$ is the unit circle and the generators $f_i, 1\leq i\leq k,$ are all monotonic, then we have
$$
 h(f)=\sum_{i=1}^k \nu_i\log |\deg(f_i)|.
$$

(2) If $X=I$ is an interval and the generators $f_i, 1\leq i\leq k,$ are all piecewise
monotonic, then we have
$$
 h(f)\leq\sum_{i=1}^k \nu_i\log \mathcal {N}(f_i).
$$
\end{proposition}

\section{An application: calculation of Friedland's entropy for $\mathbb{Z}^k$-actions}

It is well known that the concept of entropy for
$\mathbb{Z}^k$-actions was introduced by Ruelle \cite{Ruelle1} to investigate the lattice statistical mechanics. A necessary condition for this entropy to be
positive is that the generators should have infinite entropy as
single transformations.  In \cite{Friedland}, Friedland gave a new
type of entropy for $\mathbb{Z}^k$-actions (or, more
generally, $\mathbb{Z}_+^k$-actions) via the entropy of the natural shift map on the orbit space. A feature of Friedland's definition is that the entropy of a $\mathbb{Z}^k$-action is more than or equal to that of its generators.
For simplicity, we only consider Friedland's entropy of $\mathbb{Z}^k$-actions in this section. For more information and calculation of Friedland's entropy for $\mathbb{Z}_+^k$-actions, we refer to \cite{Friedland}, \cite{Geller} and \cite{Zhu14}.

 To estimate Friedland's entropy of a ${\mathbb{Z}}^k$-action, a natural way is to consider its extension. Let $\Omega$ be as in (\ref{Omega}), $\sigma$ be the left shift operator on $\Omega$ and $\Psi: \Omega\times X\longrightarrow \Omega\times X$ be the induced skew product transformation as in (\ref{Psi}). Obviously, it is an extension of $(\Omega, \sigma)$ since there is a projection
$$
\pi:\Omega\times X\longrightarrow \Omega,\;\pi(\omega,x)=\omega
$$
such that $\sigma\circ\pi=\pi\circ\Psi$. Hence by Bowen's entropy inequality (Theorem 17 of \cite{Bowen}), we have
\begin{equation}\label{Bowen}
h(\Psi)\leq h(\sigma)+\sup_{\omega\in\Omega}
h(\Psi,\pi^{-1}\omega).
\end{equation}
By \cite{Geller}, $\Psi$ is also an extension of $\sigma_{\alpha}$ since we can define a map
$$
\tilde{\pi}: \Omega\times X\longrightarrow X_{\alpha},\;\tilde{\pi}(\omega,x)=\{f_{\omega}^n(x)\}_{n\in{\mathbb{Z}}}
$$
such that $\tilde{\pi}\circ \Psi=\sigma_{\alpha} \circ \tilde{\pi}$.
Therefore, $h(\sigma_{\alpha})\le h(\Psi)$, and so by (\ref{Bowen}),
$$
h(\sigma_{\alpha})\le \log k +\sup_{\omega\in\Omega}
h(\Psi,\pi^{-1}\omega).
$$

When we consider $C^2 $ $\mathbb{Z}^k$-actions, we can get better estimations of Friedland's entropy.  We rewrite Theorem 2 in the following more explicit version.  To prove it we combine the techniques of variational principles for entropies and pressures and Birkhoff's Ergodic Theorem.

\begin{theorem}[Theorem 2]\label{Friedlandentropy}
Let $\alpha:\mathbb{Z}^k\longrightarrow C^2(M, M)$ be a $C^2$ $\mathbb{Z}^k$-action on a  $d$-dimensional closed Riemannian manifold $M$. Let $\mathcal{G}_{\alpha}$ and $\Omega$ be as in (\ref{generator}) and (\ref{Omega}) respectively, and $\Psi$ be the skew product transformation as in  (\ref{Psi}). If there is a measure with maximal entropy of $\Psi$ in the form of $\mathbf{P}_{\nu}\times\mu$, where $\mathbf{P}_{\nu}=\nu^{\mathbb{Z}}$ is the product measure of some Borel probability measure $\nu$  on $\mathcal{G}_{\alpha}$ with $\nu_i=\nu(f_i)$ and
$\mu$ is an $\alpha$-invariant measure on $M$ which is absolutely continuous with respect to the Lebesgue measure $m$. Then
the inequality (\ref{Fried1}) holds.

Moreover, if the above $\mu$ is also ergodic with respect to $\alpha$, then
\begin{equation}\label{Fried2}
h(\sigma_{\alpha})\le \max_{J\subset \{1,\cdots,s\}}\log\Big(\sum_{i=1}^k\exp\big(\displaystyle\sum_{j\in J}d_j\lambda_{i,j}\big)\Big),
\end{equation}
and the above $\nu$ can be determined by
\begin{equation}\label{measures}
\nu_i=\frac{\displaystyle\exp\big(\sum_{j\in J^*}d_j\lambda_{i,j}\big)}{\displaystyle\sum_{i=1}^k\exp\big(\displaystyle\sum_{j\in J^*}d_j\lambda_{i,j}\big)}, \;1\le i\le k,
 \end{equation}
for any $J^*\subset \{1,\cdots,s\}$ with
$\displaystyle\sum_{i=1}^k\exp\big(\displaystyle\sum_{j\in J^*}d_j\lambda_{i,j}\big)=\max_{J\subset \{1,\cdots,s\}}\sum_{i=1}^k\exp\big(\displaystyle\sum_{j\in J}d_j\lambda_{i,j}\big).$

Furthermore, if the above $\mu$ is ergodic and for each pair of generators $f_i$ and $f_j$, $1\le i\neq j\le k$, $\mu$(Coinc$(f_i, f_j))=0$, then the equality in (\ref{Fried2}) holds, i.e., the Friedland's entropy formula (\ref{Fried3}) holds.
\end{theorem}
\begin{proof}
By Abromov-Rohklin formula \cite{Abromov}, for any invariant measure of $\Psi$ in the form of $\mathbf{P}_{\nu}\times\mu$, where $\mathbf{P}_{\nu}=\nu^{\mathbb{Z}}$, we have
$$
h_{\mathbf{P}_{\nu}\times\mu}(\Psi)=h_{\mathbf{P}_{\nu}}(\sigma) +h_{\mu}(f).
$$
Since $h_{\mathbf{P}_{\nu}}(\sigma)= -\sum_{i=1}^k\nu_i\log\nu_i$, by (\ref{entropyformula1}),
\begin{equation}\label{Abromov}
h_{\mathbf{P}_{\nu}\times\mu}(\Psi)=-\sum_{i=1}^k\nu_i\log\nu_i +\int_M\max_{J\subset \{1,\cdots,s(x)\}}\sum_{j\in J}\sum_{i=1}^k \nu_i d_j(x)\lambda_{i,j}(x)d\mu(x).
\end{equation}
By the assumption on the measure with maximal entropy of $\Psi$ and the fact $h(\sigma_{\alpha})\le h(\Psi)$,  we obtain (\ref{Fried1}).

When $\mu$ is $\alpha$-ergodic, (\ref{Abromov}) becomes
\begin{equation}\label{Fried4}
h_{\mathbf{P}_{\nu}\times\mu}(\Psi)= -\sum_{i=1}^k\nu_i\log\nu_i +\max_{J\subset \{1,\cdots,s\}}\sum_{j\in J}\sum_{i=1}^k \nu_i d_j\lambda_{i,j}.
\end{equation}
For any $J\subset \{1,\cdots,s\}$, define a function
$$
\eta_J:\Omega\longrightarrow \mathbb{R}^+,\;
\eta_J(\omega)=\sum_{j\in J}d_j\lambda_{\omega, j}.
$$
Then (\ref{Fried4}) becomes
\begin{equation}\label{Fried5}
h_{\mathbf{P}_{\nu}\times\mu}(\Psi)= \max_{J\subset \{1,\cdots,s\}}\Big\{-\sum_{i=1}^k\nu_i\log\nu_i +\int_{\Omega}\eta_J d\mathbf{P}_{\nu}(\omega)\Big\}.
\end{equation}
 Since $\mathbf{P}_{\nu}\times\mu$ is a measure with maximal entropy of $\Psi$, we can apply the variational principle of $\Psi$ as follows
\begin{eqnarray}\label{pressure}
h(\Psi)&=&\sup_{\mathbf{P}_{\nu'}\times\mu} \max_{J\subset \{1,\cdots,s\}}\Big\{ -\sum_{i=1}^k\nu'_i\log\nu'_i +\int_{\Omega}\eta_J d\mathbf{P}_{\nu'}(\omega)\Big\}\notag\\
&=&\max_{J\subset \{1,\cdots,s\}}\sup_{\mathbf{P}_{\nu'}}\Big\{h_{\mathbf{P}_{\nu'}}(\sigma)+\int_{\Omega}\eta_J d\mathbf{P}_{\nu'}(\omega)\Big\}\\
&=&\max_{J\subset \{1,\cdots,s\}}P(\sigma,\eta_J),\notag
\end{eqnarray}
where the supremum is taken over all $\Psi$-invariant probability measures of the form $\mathbf{P}_{\nu'}\times\mu$ in which $\mathbf{P}_{\nu'}=\nu'^{\mathbb{Z}}$ is the product measure of some Borel probability measure $\nu'$  on $\mathcal{G}_{\alpha}$ with $\nu'_i=\nu'(f_i)$, and in the last line we use the variational principle for the pressure of $\eta_J$ with respect to $\sigma$. From Chapter 9 of \cite{Walters}, we get that,
\begin{equation}\label{pressureJ}
P(\sigma,\eta_J)=\log\Big(\sum_{i=1}^k\exp\big(\displaystyle\sum_{j\in J}d_j\lambda_{i,j}\big)\Big).
\end{equation}
Therefore, by (\ref{pressure}) and (\ref{pressureJ}), the desired inequality (\ref{Fried2}) holds. Moreover, by
Theorem 9.16 of \cite{Walters}, $\eta_J$ has a unique equilibrium state which is the product measure defined by the measure on $\mathcal{G}_{\alpha}$ which gives the element $f_i$, $1\leq i\leq k$, measure
$$
\nu_i=\frac{\displaystyle\exp\big(\sum_{j\in J}d_j\lambda_{i,j}\big)}{\displaystyle\sum_{i=1}^k\exp\big(\displaystyle\sum_{j\in J}d_j\lambda_{i,j}\big)}.
$$
So any measure $\nu$ defined by $\nu_i$ as in (\ref{measures}) satisfies that the product measure $\mathbf{P}_{\nu}\times\mu$ is a measure with maximal entropy of $\Psi$.

Furthermore, if $\mu$ is $\alpha$-ergodic and $\mu$(Coinc$(f_i, f_j))=0$ for each pair of generators $f_i,f_j,1\le i\neq j\le k$, then we conclude that $\tilde{\pi}$ is one-to-one on a set of full $\mathbf{P}_{\nu}\times\mu$ measure.  So
 \begin{equation}\label{equality2}
h_{\mathbf{P}_{\nu}\times\mu}(\Psi)=h_{\tilde{\pi}(\mathbf{P}_{\nu}\times\mu)}(\sigma_{\alpha}).
 \end{equation}
Moreover, by the variational principle for $\sigma_{\alpha}$  we have that
 \begin{equation}\label{equality3}
h_{\tilde{\pi}(\mathbf{P}_{\nu}\times\mu)}(\sigma_{\alpha})\leq h(\sigma_{\alpha}).
 \end{equation}
By (\ref{pressure}), (\ref{equality2}) and (\ref{equality3}), $h(\Psi)\leq h(\sigma_{\alpha})$. Together with the previous inequality $h(\sigma_{\alpha})\leq h(\Psi)$ we have $h(\Psi)=h(\sigma_{\alpha})$, and hence by (\ref{pressure}) and (\ref{pressureJ}), formula (\ref{Fried3}) holds.
\end{proof}

\begin{remark}
In Theorem \ref{Friedlandentropy}, we require that the invariant measure of $\Psi$ is in the form of $\mathbf{P}_{\nu}\times\mu$, we can see section 3.4 of \cite{Liu01} for the existence of such a measure for certain systems. In particular, from \cite{Zhu14}, if $X=\mathbb{T}^d$ is the $d$-dimensional torus and $\alpha$ is a linear $\mathbb{Z}^k$-action as in Example \ref{example}, then all the assumptions in Theorem \ref{Friedlandentropy} hold, so formula (\ref{Fried3}) holds, and hence,
\begin{equation}\label{entropytorus1}
 h(\sigma_{\alpha})=\max_{J\subset \{1,\cdots,s\}}\log\Bigl(\sum_{i=1}^{k}\prod_{j\in J}\lambda_{i,j}^{d_j}\Bigr),
 \end{equation}
in which $s, \lambda_{i,j}, d_j$ come from Example \ref{example}.
\end{remark}

\end{document}